\documentclass[final,leqno]{siamltex}

\usepackage{hyperref}
\usepackage{graphicx}
\usepackage{bm}
\usepackage{amsmath}
\usepackage{xcolor}
\usepackage{amssymb}
\usepackage{amsfonts,stackengine}
\usepackage{marginnote}
\usepackage{fancybox}
\usepackage{color}
\usepackage{bm}
\usepackage{cancel}
\usepackage{wrapfig}
\usepackage[notcite, notref]{showkeys} 
\usepackage{lipsum}
\usepackage{epstopdf}
\usepackage{stmaryrd}
\usepackage{mathrsfs}
\usepackage{psfrag}
\usepackage{caption}
\usepackage{mathtools}
\usepackage{amsopn}
\usepackage{placeins}
\usepackage[normalem]{ulem}
\usepackage{enumitem}

\newcommand{\EO}[1]{\noindent{\textcolor{blue}{#1}}}
\newtheorem{remark}{Remark}[section]

\newcommand{\mbRn}{{\mathbb{R}^n}}
\newcommand{\T}{\mathscr{T}}
\newcommand{\V}{\mathbb{V}}

\newcommand{\omg}{{\Omega}}
\newcommand{\omgc}{{\Omega^c}}

\title{
Fractional semilinear optimal control: optimality conditions, convergence, and error analysis\thanks{Submitted to the editors \today. }}
\author{Enrique Ot\'arola\thanks{Departamento de Matem\'atica, Universidad T\'ecnica Federico Santa Mar\'ia, Valpara\'iso, Chile ({\tt enrique.otarola@usm.cl}).}
}

\begin{document}
\maketitle

\begin{abstract}
We adopt the integral definition of the fractional Laplace operator and analyze an optimal control problem for a fractional semilinear elliptic partial differential equation (PDE); control constraints are also considered. We establish the well-posedness of fractional semilinear elliptic PDEs and analyze regularity properties and suitable finite element discretizations. Within the setting of our optimal control problem, we derive the existence of optimal solutions as well as first and second order optimality conditions; regularity estimates for the optimal variables are also analyzed. We devise a fully discrete scheme that approximates the control variable with piecewise constant functions; the state and adjoint equations are discretized with continuous piecewise linear finite elements. We analyze convergence properties of discretizations and derive a priori error estimates.
\end{abstract}

\begin{keywords}
optimal control problem, fractional diffusion, integral fractional Laplacian, regularity estimates, finite elements, convergence, a priori error estimates.
\end{keywords}

\begin{AMS}
35R11,    
49J20,     
49M25,    
65K10,    
65N15,    
65N30.    
\end{AMS}

\section{Introduction}
In this work we are interested in the analysis and discretization of a distributed optimal control problem for a \emph{fractional}, \emph{semilinear}, and \emph{elliptic} partial differential equation (PDE). To make matters precise, we let $\Omega \subset \mathbb{R}^n$ be an open and bounded domain in $\mathbb{R}^n$ $(n \in \{2,3\})$ with Lipschitz boundary $\partial \Omega$; additional regularity requirements on $\partial \Omega$ will be imposed in the course of our regularity and convergence rate analyses ahead. Let us introduce the cost functional
\begin{equation}
 \label{eq:cost_functional}
 J(u,z) := \int_{\Omega} L (x,u(x)) \mathrm{d}x + \frac{\alpha}{2} \int_{\Omega} |z(x)|^2  \mathrm{d}x,
\end{equation}
where $L: \Omega \times \mathbb{R} \rightarrow \mathbb{R}$ denotes a Carath\'eodory function of class $C^2$ with respect to the second variable and $\alpha > 0$ corresponds to the so-called regularization parameter. Further assumptions on $L$ will be deferred until section \ref{sec:assumption}. In this work, we shall be concerned with the following PDE-constrained optimization problem: Find
$
 \min  J(u,z)
$
subject to the \emph{fractional, semilinear}, and \emph{elliptic} PDE
\begin{equation}
\label{eq:state_equation}
(-\Delta)^s u + a(\cdot,u) = z \textrm{ in } \Omega, \qquad u = 0 \textrm{ in } \Omega^c,
\end{equation}
and the \emph{control constraints}
$
 \mathfrak{a} \leq z(x) \leq \mathfrak{b}
$
for a.e.~$x \in \Omega$. Here, $\omgc=\mbRn\setminus \omg$. The control bounds $\mathfrak{a},\mathfrak{b} \in \mathbb{R}$ are such that $\mathfrak{a} < \mathfrak{b}$. Assumptions on the nonlinear function $a$ will be deferred until section \ref{sec:assumption}. We will refer to the previously defined PDE-constrained optimization problem as the \emph{fractional semilinear optimal control problem}.

For smooth functions $w: \mathbb{R}^n \rightarrow \mathbb{R}$, there are several equivalent definitions of the fractional Laplace operator $(-\Delta)^s$ in $\mathbb{R}^n$ \cite{MR3613319}. Indeed, $(-\Delta)^s$ can be naturally defined by means of the following pointwise formula:
\begin{equation}
 (-\Delta)^s w(x) = C(n,s) \, \mathrm{ p.v. } \int_{\mathbb{R}^n} \frac{w(x) - w(y)}{|x-y|^{n+2s}} \mathrm{d}y,
 \qquad
 C(n,s) = \frac{2^{2s} s \Gamma(s+\frac{n}{2})}{\pi^{n/2}\Gamma(1-s)},
 \label{eq:integral_definition}
\end{equation}
where $\textrm{p.v.}$ stands for the Cauchy principal value and $C(n,s)$ is a positive normalization constant that depends only on $n$ and $s$. Equivalently, $(-\Delta)^s$ can be defined via Fourier transform: $\mathcal{F}( (-\Delta)^s w) (\xi) = | \xi |^{2s} \mathcal{F}(w) (\xi)$. A proof of the equivalence of these two definitions can be found in \cite[section 1.1]{Landkof}. In addition to these two definitions, several other \emph{equivalent definitions} of $(-\Delta)^s$ in $\mathbb{R}^n$ are available in the literature \cite{MR3613319}. Regarding \emph{equivalence}, the scenario in bounded domains is \emph{substantially different}. For functions supported in $\bar \Omega$, we may utilize the integral representation \eqref{eq:integral_definition} to define $(-\Delta)^s$. This gives rise to the so-called \emph{restricted} or \emph{integral} fractional Laplacian. Notice that we have materialized a zero Dirichlet condition by restricting the operator to acting only on functions that are zero outside $\Omega$. We must immediately mention that in bounded domains, and in addition to the \emph{restricted} or \emph{integral} fractional Laplacian there are, at least, two other \emph{nonequivalent} definitions of nonlocal operators related to the fractional Laplacian: the \emph{regional} fractional Laplacian and the \emph{spectral} fractional Laplacian; see \cite[Section 2]{MR3393253} and \cite[Section 6]{MR3503820} for details. In this work, we adopt the \emph{restricted} or \emph{integral} definition of the fractional Laplace operator $(-\Delta)^s$, which, from now on, we shall simply refer to as the \emph{integral fractional Laplacian}. 

During the very recent past, there has been considerable progress in the design and analysis of solution techniques for linear problems involving fractional diffusion. We refer the interested reader to \cite{MR3893441,acta_marta} for a complete overview of the available results and limitations. In contrast to these advances, the numerical analysis of PDE-constrained optimization problems involving $(-\Delta)^s$ has been less explored. Restricting ourselves to problems that consider the spectral definition, we mention \cite{MR3429730,MR4015150,MR3702421} within the linear--quadratic scenario, \cite{MR3850351} for optimization with respect to order, \cite{MR4066856,MR3739306} for sparse PDE-constrained optimization, and \cite{MR3939497} for bilinear optimal control. We also mention \cite{semilinear}, where the authors analyze, at the continuous level, a semilinear optimal control problem for the spectral and integral fractional Laplacian. Concerning the integral fractional Laplacian, it seems that the results are even scarcer; the linear--quadratic case has been recently analyzed in \cite{MR3990191,glusaotarola}. We conclude this paragraph by mentioning \cite{MR3596859,MR3158780} for discretizations of optimal control problems involving suitable nonlocal operators and \cite{MR4202972} for a related fractional optimal control problem.

In addition to this exposition being the first one that studies numerical schemes for semilinear optimal control problems involving the \emph{integral} fractional Laplacian, the analysis itself comes with its own set of difficulties. Overcoming them has required us to provide several results. Let us briefly detail some of them:
\begin{enumerate}
\label{i}
\item[(i)] \emph{Fractional PDEs:} Let $s \in (0,1)$, $n \geq 2$, $r>n/2s$, and  $z \in L^r(\Omega)$. We show that \eqref{eq:state_equation} is well-posed for $a = a(x,u)$ being a Carath\'eodory function, monotone increasing in $u$, satisfying \eqref{eq:assumption_on_phi_state_equation} and $a(\cdot,0) \in L^r(\Omega)$ (Theorem \ref{thm:stata_equation_well_posedness}).

\label{ii}
\item[(i)] \emph{FEM discretizations:} We prove convergence of finite element discretizations on Lipschitz polytopes and obtain error estimates on smooth domains; the latter under additional assumptions on $a$ and the underlying forcing term that guarantee the regularity estimates of Theorem \ref{thm:regularity_space_state_equation}; see section \ref{sec:fem}.

\label{iii}  
\item[(iii)] \emph{Existence of an optimal control:} Assuming that, in addition, $L=L(x,u)$ is a Carath\'eodory function and $a$ and $L$ satisfy \eqref{eq:aux_for_existence}, we show that our control problem admits at least a solution; see Theorem \ref{thm:existence_control}.

\label{iv}
\item[(iv)] \emph{Optimality conditions:} Let $n \in \{2,3\}$ and $s > n/4$. Under additional assumptions on $a$ and $L$, we derive second order necessary and sufficient optimality conditions with a minimal gap; see Section \ref{sec:2nd_order}.

\label{v} 
\item[(v)] \emph{Regularity estimates:} Let $n \geq 2$ and $s \in (0,1)$. We obtain regularity properties for optimal variables: $\bar u,\bar p, \bar{z} \in H^{s+1/2-\epsilon}(\Omega)$,
where $\epsilon >0$ is arbitrarily small; see Theorem \ref{thm:regularity_space}.

\label{vi}
\item[(vi)] \emph{Convergence of discretization and error estimates:} Let $n \geq 2$ and $s \in (0,1)$. 
We prove that global solutions of discrete optimal control problems converge to a global solution of the continuous one and that strict local continuous solutions can be approximated by local discrete ones; see Theorems \ref{thm:convergence} and \ref{thm:convergence_local_minima}.
When $n \in \{2,3\}$ and $s>n/4$, we derive error estimates; see Theorem \ref{thm:error_estimate_control}. To obtain these results we have assumed that solutions to finite element discretizations of \eqref{eq:state_equation} are uniformly bounded in $L^{\infty}(\Omega)$.
\end{enumerate}

Over the last 20 years, several contributions have delineated the numerical analysis of semilinear optimal control problems. Without a doubt, these studies have paved the way for the achievement of the aforementioned results. In particular, we have followed \cite{MR2583281}, for the analysis of \eqref{eq:state_equation} and the optimal control problem, \cite{MR3586845}, for deriving second order optimality conditions, and \cite{MR2350349,MR3023751,MR3586845}, for analyzing convergence properties and deriving error estimates.

The rest of the paper is organized as follows. In section \ref{sec:state_equation}, we analyze the fractional state equation \eqref{eq:state_equation}. A complete study of the fractional semilinear optimal control problem is presented in section \ref{sec:optimal_control_problem}. In sections \ref{sec:fem} and \ref{sec:fem_adjoint}, we study finite element discretizations for \eqref{eq:state_equation} and the so-called adjoint equation, respectively. Section \ref{sec:fem_control} is dedicated to the analysis of finite element discretizations for the fractional semilinear optimal control problem: convergence and error estimates.

\section{Notation and preliminaries}\label{sec:notation}
Let us begin by presenting the main notation and assumptions we shall operate under. For $n \geq 2$, we let $\Omega \subset \mathbb{R}^n$ be an open and bounded domain with Lipschitz boundary $\partial \Omega$; we will impose additional assumptions on $n$ and $\partial \Omega$ when needed. We will denote by $\Omega^c$ the complement of $\Omega$.  If $\mathcal{X}$ and $\mathcal{Y}$ are normed spaces, we write $\mathcal{X}  \hookrightarrow \mathcal{Y}$ to denote that $\mathcal{X}$ is continuously embedded in $\mathcal{Y}$. Let $\{ x_n \}_{n=1}^{\infty}$ be a sequence in $\mathcal{X}$. We will denote by $x_n \rightarrow x$ and $x_n \rightharpoonup x$ the  strong and weak convergence, respectively, of $\{ x_n \}_{n=1}^{\infty}$ to $x$. The relation ${\sf a} \lesssim {\sf b}$ indicates that ${\sf a} \leq C {\sf b}$, with a positive constant $C$ that does not depend on either ${\sf a}$, ${\sf b}$, or the discretization parameters, but it might depend on $s$, $n$, and $\Omega$. The value of $C$ might change at each occurrence.

\subsection{Assumptions}\label{sec:assumption}

We will operate under the following assumptions on $a$ and $L$. We must, however, immediately mention that some of the results obtained in this work are valid under less restrictive requirements; when possible we explicitly mention the assumptions on $a$ and $L$ that are needed to obtain a particular result.

\begin{enumerate}[label=(A.\arabic*)]
\item \label{A1} $a:\Omega\times \mathbb{R}\rightarrow  \mathbb{R}$ is a Carath\'eodory function of class $C^2$ with respect to the second variable and $a(\cdot,0)\in L^{r}(\Omega)$ for $r>n/2s$.

\item \label{A2} $\frac{\partial a}{\partial u}(x,u)\geq 0$ for a.e.~$x\in\Omega$ and for all $u \in \mathbb{R}$.

\item \label{A3} For all $\mathfrak{m}>0$, there exists a positive constant $C_{\mathfrak{m}}$ such that 
\begin{equation*}
\sum_{i=1}^{2}\left|\frac{\partial^{i} a}{\partial u^{i} }(x,u)\right|\leq C_{\mathfrak{m}},
\qquad
\left|\frac{\partial^{2} a}{\partial u^{2} }(x,v)- \frac{\partial^{2} a}{\partial u^{2} }(x,w)\right|\leq C_{\mathfrak{m}} |v-w|
\end{equation*}
for a.e.~$x\in \Omega$ and $u,v,w \in [-\mathfrak{m},\mathfrak{m}]$.
\end{enumerate}

\begin{enumerate}[label=(B.\arabic*)]
\item \label{B1} $L: \Omega \times \mathbb{R} \rightarrow \mathbb{R}$ is a Carath\'eodory function of class $C^2$ with respect to the second variable and $L(\cdot,0) \in L^1(\Omega)$.

\item \label{B2} For all $\mathfrak{m}>0$, there exist $\psi_{\mathfrak{m}}, \phi_{\mathfrak{m}} \in L^{r}(\Omega)$, with $r>n/2s$, such that
\begin{equation*}
\left|
\frac{\partial L}{\partial u}(x,u)
\right|
\leq \psi_{\mathfrak{m}}(x),
\qquad
\left|
\frac{\partial^2 L}{\partial u^2}(x,u)
\right|
\leq \phi_{\mathfrak{m}}(x),
\end{equation*}
for a.e.~$x\in \Omega$ and $u \in [-\mathfrak{m},\mathfrak{m}]$.
\end{enumerate}

The following assumptions are particularly needed to derive regularity estimates:

\begin{enumerate}[label=(C.\arabic*)]
\item \label{C1} $a(\cdot,0) \in L^2(\Omega) \cap H^{\frac{1}{2} - s - \epsilon}(\Omega)$ and $\frac{\partial a}{\partial u}(\cdot,0) \in H^{\beta}(\Omega)$ for every $\beta < \tfrac{1}{2}$.

\item \label{C2} For every $\mathfrak{m}>0$ and $u \in [-\mathfrak{m},\mathfrak{m}]$, $\frac{\partial L}{\partial u}(\cdot,u) \in L^2(\Omega) \cap H^{\frac{1}{2} - s - \epsilon}(\Omega)$.
\end{enumerate}
In \ref{C1} and \ref{C2}, $\epsilon>0$ denotes an arbitrarily small positive constant.

\subsection{Function spaces}
For any $s \geq 0$, we define $H^s(\mathbb{R}^n)$, the Sobolev space of order $s$ over $\mathbb{R}^n$, by \cite[Definition 15.7]{Tartar}
\[
 H^s(\mathbb{R}^n) := \left \{ v \in L^2(\mathbb{R}^n): (1+|\xi|^2)^{s/2} \mathcal{F}(v) \in L^2(\mathbb{R}^n)\right \}.
\]
With the space $H^s(\mathbb{R}^n)$ at hand, we define $\tilde H^s(\Omega)$ as the closure of $C_0^{\infty}(\Omega)$ in $H^s(\mathbb{R}^n)$. This space can be equivalently characterized by \cite[Theorem 3.29]{McLean}
\begin{equation}
\tilde H^s(\Omega) = \{v|_{\Omega}: v \in H^s(\mathbb{R}^n), \textrm{ supp } v \subset \overline\Omega\}.
\end{equation}
When $\partial \Omega$ is Lipschitz $\tilde H^s(\Omega)$ is equivalent to $\mathbb{H}^s(\Omega)=[L^2(\Omega),H_0^1(\Omega)]_s$, the real interpolation between $L^2(\Omega)$ and $H_0^1(\Omega)$ for $s \in (0,1)$ and to $H^s(\Omega) \cap H_0^1(\Omega)$ for $s \in (1,3/2)$ \cite[Theorem 3.33]{McLean}. We denote by $H^{-s}(\Omega)$ the dual space of $\tilde H^s(\Omega)$ and by $\langle \cdot, \cdot \rangle$ the duality pair between these two spaces. We define the bilinear form
\begin{equation}
\label{eq:bilinear_form}
 \mathcal{A}(v,w) = \frac{C(n,s)}{2} \iint_{\mathbb{R}^n \times \mathbb{R}^n} \frac{( v(x) - v(y) ) (w(x)-w(y))}{|x-y|^{n+2s}} \mathrm{d}x \mathrm{d}y,
\end{equation}
and denote by $\| \cdot \|_s$ the norm that $ \mathcal{A}(\cdot,\cdot)$ induces, which is just a multiple of the $H^s(\mathbb{R}^n)$-seminorm:
$
 \| v \|_s =  \sqrt{ \mathcal{A}(v,v)} = \mathfrak{C}(n,s) |v|_{H^s(\mathbb{R}^n)},
$
where $\mathfrak{C}(n,s) = \sqrt{C(n,s)/2}$.

We will repeatedly use the following continuous embedding: $H^s(\Omega) \hookrightarrow L^{\mathfrak{q}}(\Omega)$ for $1 \leq \mathfrak{q} \leq 2n/(n-2s)$ \cite[Theorem 7.34]{MR2424078}; observe that $n>2s$. If $\mathfrak{q} < 2n/(n-2s)$ the embedding $H^s(\Omega)  \hookrightarrow L^{\mathfrak{q}}(\Omega)$ is compact \cite[Theorem 6.3]{MR2424078}.

\section{The state equation}
\label{sec:state_equation}
Let $f \in H^{-s}(\Omega)$ be a forcing term. In this section, we analyze the following \emph{fractional}, \emph{semilinear}, and \emph{elliptic} PDE:
\begin{equation}
 \label{eq:weak_semilinear_pde}
  \mathcal{A}(u,v)  +  \langle a(\cdot,u),v \rangle = \langle f , v \rangle 
  \quad \forall v \in \tilde H^{s}(\Omega).
\end{equation}
Here, $a = a(x,u) : \Omega \times \mathbb{R} \rightarrow \mathbb{R}$ denotes a Carath\'eodory function that is monotone increasing in $u$. In addition, we assume that, for every $\mathfrak{m}>0$, there exits
\begin{equation}
\varphi_{\mathfrak{m}} \in L^{\mathfrak{t}}(\Omega): 
\quad
|a(x,u)| \leq | \varphi_{\mathfrak{m}}(x)|
~\textrm{a.e.}~x \in \Omega,~u \in [-\mathfrak{m},\mathfrak{m}],
\quad 
\mathfrak{t} =2n/(n+2s).
\label{eq:assumption_on_phi_state_equation}
\end{equation}

We present the following existence and uniqueness result.

\begin{theorem}[well-posedness of fractional and semilinear PDEs]
Let $n \geq 2$,  $s \in (0,1)$, and $r>n/2s$. Let $\Omega \subset \mathbb{R}^n$ be an open and bounded domain with Lipschitz boundary. If $f \in L^{r}(\Omega)$, $a$ satisfies \eqref{eq:assumption_on_phi_state_equation}, and $a(\cdot,0) \in L^r(\Omega)$, 
then problem \eqref{eq:weak_semilinear_pde} admits a unique solution $u \in \tilde H^s(\Omega) \cap L^{\infty}(\Omega)$. In addition, we have the estimate
\begin{equation}
 | u |_{H^s(\mathbb{R}^n)} + \| u \|_{L^{\infty}(\Omega)} \lesssim \| f - a(\cdot,0) \|_{L^{r}(\Omega)},
 \label{eq:stability}
\end{equation}
with a hidden constant that is independent of $u$, $a$, and $f$.
\label{thm:stata_equation_well_posedness}
\end{theorem}
\begin{proof}
We proceed in four steps.

\noindent \framebox{1} Let us assume, for the moment, that, in addition, there exists $\varphi \in L^{\mathfrak{t}}(\Omega)$ such that $|a(x,u)| \leq |\varphi(x)|$ for a.e.~$x \in \Omega$ and $u \in \mathbb{R}$ and that
$a(\cdot,0) = 0$. Define the mapping 
\[
\mathfrak{A}: \tilde H^s(\Omega) \rightarrow H^{-s}(\Omega):
\quad
\langle \mathfrak{A} u,v \rangle =   \mathcal{A}(u,v)  +  \langle a(\cdot,u),v \rangle
\quad
\forall v \in \tilde H^{s}(\Omega).
\]
Since $\mathcal{A}$ is bilinear, continuous, and coercive on $\tilde H^s(\Omega) \times \tilde H^s(\Omega)$ and $a = a(x,u)$ is globally bounded and monotone increasing in $u$, it is immediate that $\mathfrak{A}$ is \emph{well-defined}, \emph{monotone}, and \emph{coercive}. In addition, since $a = a(x,u)$ is continuous in $u$ for a.e.~$x \in \Omega$, dominated convergence yields the \emph{hemicontinuity} of $\mathfrak{A}$. Existence and uniqueness of $u \in \tilde H^s(\Omega)$ follows from the main theorem on monotone operators \cite[Theorem 26.A]{MR1033498}, \cite[Theorem 2.18]{MR3014456}. Set $v = u$ \eqref{eq:weak_semilinear_pde} to obtain $ | u |_{H^s(\mathbb{R}^n)} \lesssim \| f \|_{H^{-s}(\Omega)}$.

\noindent \framebox{2}  Define, for $k >0$, $v_k$ by
$
v_k(x) = u(x) - k \textrm{ if } u(x) \geq k,
$
$
v_k(x) = 0 \textrm{ if } |u(x)| < k,
$
and
$
v_k(x) = u(x) + k \textrm{ if } u(x) \leq- k.
$
We also define the set 
\[
\Omega(k):= \{ x \in \Omega: |u(x)| \geq k\}.
\]
Since $a = a(x,u)$ is monotone increasing in $u$ and $a(\cdot,0) = 0$, we have $\langle a(\cdot,u), v_k \rangle = \int_{\Omega} a(x,u(x)) v_k(x) \mathrm{d}x \geq 0$. This yields $\mathcal{A}(u,v_k) \leq \langle f, v_k \rangle$. The relations and inequalities (2.22)--(2.30) in \cite{MR3745164} reveal that $\mathcal{A}(v_k,v_k) \leq \mathcal{A}(u,v_k)$. We can thus obtain $ \| v_k \|^2_{s}  = \mathcal{A}(v_k,v_k) \leq \langle f, v_k \rangle$. Define $\mathfrak{q}:=2n/(n-2s)$. Thus, for $t<2n/(n-2s)$,
\[
\| v_k \|_{L^{\mathfrak{q}}(\Omega)}^2 \lesssim | v_k |^2_{H^s(\mathbb{R}^n)} \lesssim \| v_k \|_{L^{\mathfrak{q}}(\Omega)} | \Omega(k)|^{\frac{1}{t}} \|f \|_{L^{r}(\Omega)},
\quad
\mathfrak{q}^{-1} + r^{-1} + t^{-1} = 1.
\]
On the other hand, $\| v_k \|_{L^{\mathfrak{q}}(\Omega)}^{\mathfrak{q}} = \int_{\Omega(k)} |v_k(x)|^{\mathfrak{q}} \mathrm{d}x = \int_{\Omega(k)} | |u(x)|-k |^{\mathfrak{q}} \mathrm{d}x$. Let $h>k$, then $\Omega(h) \subset \Omega(k)$ and $ \int_{\Omega(k)} | |u(x)|-k |^{\mathfrak{q}} \mathrm{d}x \geq (h-k)^{\mathfrak{q}} |\Omega(h)|$. Consequently,
\[
(h-k) |\Omega(h)|^{\frac{1}{\mathfrak{q}}} \leq \| v_k \|_{L^{\mathfrak{q}}(\Omega)} \lesssim (  | \Omega(k)|^{\frac{1}{\mathfrak{q}}})^{\frac{\mathfrak{q}}{t}} \|f \|_{L^{r}(\Omega)}.
\]
Since $\mathfrak{q}/t > 1$, an application of \cite[Lemma B.1]{MR1786735} yields the existence of $\mathfrak{h} > 0$ such that $|\Omega(\mathfrak{h}) | = 0$, which implies that $u \in L^{\infty}(\Omega)$ and $\| u \|_{L^{\infty}(\Omega)} \lesssim \| f\|_{L^r(\Omega)}$.

\noindent \framebox{3} We relax the assumption of step 1. Define, for $k >0$, $a_k$ by
$
a_k(x,u) = a(x,k) \textrm{ if } u > k,
$
$
a_k(x,u) = a(x,u) \textrm{ if } |u| \leq k,
$
and
$
a_k(x,u) = a(x,-k) \textrm{ if } u < - k.
$
In view of \eqref{eq:assumption_on_phi_state_equation}, there exists $\varphi_k \in L^{\mathfrak{t}}(\Omega)$ such that, for a.e.~$x\in\Omega$ and $u \in \mathbb{R}$, $|a_k(x,u)| \leq |\varphi_k(x)|$. We thus invoke the arguments of the previous steps to guarantee the existence of a unique solution $u$ to problem \eqref{eq:weak_semilinear_pde} with $a$ replaced by $a_k$. In addition, we have $| u |_{H^s(\mathbb{R}^n)} + \| u \|_{L^{\infty}(\Omega)} \leq c \| f \|_{L^{r}(\Omega)}$ with $c>0$ being independent of $a_k$ and thus of $k$. Choose $k> c \| f \|_{L^{r}(\Omega)}$ so that  $a_k(x,u(x))= a(x,u(x))$ for a.e.~$x \in \Omega$. Consequently, $u$ solves \eqref{eq:weak_semilinear_pde}. Uniqueness of solutions follows from the monotonicity of $a$.

\noindent \framebox{4} We remove the condition $a(\cdot,0) = 0$ by replacing $a(\cdot,u)$ by $a(\cdot,u) - a(\cdot,0)$.
\end{proof}

\section{The optimal control problem}
\label{sec:optimal_control_problem}

In this section, we analyze the following weak version of the \emph{fractional semilinear optimal control problem}: Find
\begin{equation}\label{eq:min}
\min \{ J(u,z): (u,z) \in \tilde H^s(\Omega) \times \mathbb{Z}_{ad}\}
\end{equation}
subject to the \emph{fractional, semilinear,} and \emph{elliptic} state equation
\begin{equation}\label{eq:weak_st_eq}
\mathcal{A}( u, v)+(a(\cdot,u),v)_{L^2(\Omega)}=(z,v)_{L^2(\Omega)} \quad \forall v \in \tilde H^s(\Omega).
\end{equation}
Here, $ \mathbb{Z}_{ad}:=\{ v \in L^2(\Omega):  \mathfrak{a} \leq v(x) \leq \mathfrak{b}~\text{a.e.}~x \in \Omega \}$ and $\mathfrak{a},\mathfrak{b} \in \mathbb{R}$ are such that $\mathfrak{a} < \mathfrak{b}$.

Let $r>n/2s$ and $a = a(x,u): \Omega \times \mathbb{R} \rightarrow \mathbb{R}$ be a monotone increasing in $u$ Carath\'eodory function satisfying \eqref{eq:assumption_on_phi_state_equation} and 
$a(\cdot,0) \in L^r(\Omega)$. Within this setting, Theorem \ref{thm:stata_equation_well_posedness} guarantees the existence of a unique solution $u$ to problem \eqref{eq:weak_st_eq}. We thus introduce the control to state map $\mathcal{S}: L^{r}(\Omega) \rightarrow \tilde H^s(\Omega) \cap L^{\infty}(\Omega)$ which, given a control $z$, associates to it the unique state $u$ that solves \eqref{eq:weak_st_eq}. With $\mathcal{S}$ at hand, we also introduce the reduced cost functional $j:\mathbb{Z}_{ad} \rightarrow \mathbb{R}$ by the relation $j(z)=J(\mathcal{S}z,z)$.

\subsection{Existence of optimal controls}
The existence of an optimal state--control pair $(\bar{u},\bar{z})$ is as follows.

\begin{theorem}[existence of an optimal pair]\label{thm:existence_control}
Let $n \geq 2$, $s \in (0,1)$ and $r>n/2s$. Let $a = a(x,u) : \Omega \times \mathbb{R} \rightarrow \mathbb{R}$ be a Carath\'eodory function that is monotone increasing in $u$. Let $L = L(x,u) : \Omega \times \mathbb{R} \rightarrow \mathbb{R}$ be a Carath\'eodory function. Assume that, for every $\mathfrak{m}>0$, there exist $\varphi_{\mathfrak{m}} \in L^{r}(\Omega)$, with $r>n/2s$, and $\psi_{\mathfrak{m}} \in L^1(\Omega)$ such that
\begin{equation}
|a(x,u)| \leq \varphi_{\mathfrak{m}}(x), \quad |L(x,u)| \leq \psi_{\mathfrak{m}}(x), \quad a.e.~x \in \Omega,~ u \in [-\mathfrak{m},\mathfrak{m}].
\label{eq:aux_for_existence}
\end{equation}
Thus, \eqref{eq:min}--\eqref{eq:weak_st_eq} admits at least one solution $(\bar{u},\bar{z}) \in \tilde H^s(\Omega) \cap L^{\infty}(\Omega) \times \mathbb{Z}_{ad}$.
\end{theorem}
\begin{proof}
Let $\{ (u_k,z_k) \}_{k=1}^{\infty}$ be a minimizing sequence, i.e., for $k \in \mathbb{N}$, $z_k \in \mathbb{Z}_{ad}$ and $u_k = \mathcal{S}z_k \in \tilde H^s(\Omega)$ are such that $J(u_k,z_k) \rightarrow \mathfrak{j}:= \inf \{ J(\mathcal{S}z,z): z \in \mathbb{Z}_{ad} \}$ as $k \uparrow \infty$. Since $\mathbb{Z}_{ad}$ is bounded in $L^{\infty}(\Omega)$, there exits a nonrelabeled subsequence $\{ z_k \}_{k=1}^{\infty}$ such that $z_k \mathrel{\ensurestackMath{\stackon[1pt]{\rightharpoonup}{\scriptstyle\ast}}} \bar{z}$ in $L^{\infty}(\Omega)$ as $k \uparrow \infty$. On the other hand, since, for every $k \in \mathbb{N}$, $z_k \in \mathbb{Z}_{ad}$, Theorem \ref{thm:stata_equation_well_posedness} yields the existence of $\mathfrak{m}>0$ such that $|u_k(x)| \leq \mathfrak{m}$ for a.e.~$x \in \Omega$ and $k \in \mathbb{N}$. This implies that $\{a(\cdot,u_k)\}_{k=1}^{\infty}$ is bounded in $L^r(\Omega)$. We can thus conclude the existence of a nonrelabeled subsequence $\{ u_k \}_{k=1}^{\infty}$ such that $u_k \rightharpoonup \bar{u}$ in $\tilde H^s(\Omega)$ and $u_k \rightarrow \bar{u}$ in $L^2(\Omega)$ as $k \uparrow \infty$; $\bar u$ is the natural candidate for the desired optimal state.

We now observe that, for $k \in \mathbb{N}$, $u_k \in \tilde H^s(\Omega) \cap L^{\infty}(\Omega)$ solves
\begin{equation}
\mathcal{A}(u_k,v) + \langle a(\cdot,u_k),v \rangle = \langle z_k ,v \rangle \quad \forall v \in \tilde H^s(\Omega).
\label{eq:u_n}
\end{equation}
Since there exists $M>0$ such that $|u_k(x)| \leq M$ for a.e.~$x \in \Omega$ and $k \in \mathbb{N}$ and the set $\mathfrak{M}:=\{ v \in L^{r}(\Omega): |v(x)|\leq M ~\mathrm{a.e}.~x \in \Omega\}$ is weakly sequentially closed, we conclude that $\bar u \in \mathfrak{M}$. We can thus invoke \eqref{eq:aux_for_existence} and the Lebesgue dominated convergence theorem
to obtain $\| a(\cdot,\bar u) - a(\cdot,u_k) \|_{L^{r}(\Omega)} \rightarrow 0$ as $k \uparrow \infty$. In view of the previous convergence results, passing to the limit in \eqref{eq:u_n} yields $\bar u = \mathcal{S} \bar z$.

On the other hand, the map $L^2(\Omega) \ni v \mapsto \| v\|_{L^2(\Omega)}^2 \in \mathbb{R}$ is continuous and convex; it is thus weakly lower continuous. Consequently,
\[
\mathfrak{j} = \lim_{k \uparrow \infty} J(u_k,z_k) = \int_{\Omega} L(x,\bar u(x) ) \mathrm{d}x + \liminf_{k \uparrow \infty} \frac{\alpha}{2} \| z_k \|^2_{L^2(\Omega)} \geq J(\bar u, \bar z).
\]
The Lebesgue dominated convergence theorem combined with \eqref{eq:aux_for_existence} and the the fact that $ u_k \rightarrow \bar u$ in $L^2(\Omega)$, as $k \uparrow 0$, yield
$\left| \int_{\Omega} [ L(x,\bar u(x)) - L(x,u_k(x)) ] \mathrm{d}x \right| 
\rightarrow 0$ as $k \uparrow 0$.
\end{proof}

\begin{remark}[assumptions on $a$]
\rm
To obtain the result of Theorem \ref{thm:existence_control} we have assumed \eqref{eq:aux_for_existence}. Observe that \eqref{eq:assumption_on_phi_state_equation} can be guaranteed because $n/2s>2n/(n+2s)$.
\end{remark}

\subsection{First order necessary optimality conditions}\label{sec:1st_order}
In this section, we analyze differentiability properties for the control to state map $\mathcal{S}$ and derive first order necessary optimality conditions. Since the optimal control problem \eqref{eq:min}--\eqref{eq:weak_st_eq} is not convex, we analyze optimality conditions in the context of local solutions.

We begin by precisely introducing the concept of local minimum. Let $q \in [1,\infty)$ and $\epsilon >0$. We denote by $B_{\epsilon}(\bar z)$ the closed ball in $L^q(\Omega)$ of radius $\epsilon$ centered at $\bar z$.

\begin{definition}[local minimum] 
Let $q \in [1,\infty)$. We say that $\bar z \in \mathbb{Z}_{ad}$ is a local minimum, or locally optimal, in $L^{q}(\Omega)$ for \eqref{eq:min}--\eqref{eq:weak_st_eq} if there exists $\epsilon >0$ such that 
$j(\bar z) \leq j(z)$ for every $z \in B_{\epsilon}(\bar z) \cap \mathbb{Z}_{ad}$. 
\end{definition}

\begin{remark}[local optimality in $L^2(\Omega) \implies$ local optimality in $L^q(\Omega)$]
\rm
Since $\mathbb{Z}_{ad}$ is bounded in $L^{\infty}(\Omega)$, it can be proved that if $\bar z \in \mathbb{Z}_{ad}$ is a (strict) local minimum in $L^2(\Omega)$, then $\bar z \in \mathbb{Z}_{ad}$ is a (strict) local minimum in $L^q(\Omega)$ \cite[Section 5]{MR3586845}.
\label{rk:local_minimum}
\end{remark}

In what follows, we will operate in $L^2(\Omega)$ regarding local optimally.

\begin{theorem}[differentiability properties of $\mathcal{S}$]
\label{thm:properties_C_to_S}
Let $n \geq 2$, $s \in (0,1)$, and $r > n/2s$. Assume that \textnormal{\ref{A1}}, \textnormal{\ref{A2}}, and \textnormal{\ref{A3}} hold. Then, the control to state map $\mathcal{S}: L^{r}(\Omega) \rightarrow \tilde H^{s}(\Omega)\cap L^{\infty}(\Omega)$ is of class $C^2$. In addition, if $z,w \in L^{r}(\Omega)$, then $\phi = \mathcal{S}'(z) w \in \tilde H^s(\Omega)  \cap L^{\infty}(\Omega)$ corresponds to the unique solution to the problem
\begin{equation}\label{eq:aux_adjoint}
\mathcal{A}(\phi,v)+\left(\tfrac{\partial a}{\partial u}(\cdot,u)\phi, v \right)_{L^2(\Omega)}=(w,v)_{L^2(\Omega)} \quad \forall v \in \tilde H^s(\Omega),
\end{equation}
where $u = \mathcal{S}z$. If $w_1,w_2\in L^{r}(\Omega)$, then $\psi=\mathcal{S}''(z)(w_1,w_2)\in \tilde H^s(\Omega) \cap L^{\infty}(\Omega)$ corresponds to the unique solution to 
\begin{equation}\label{eq:aux_adjoint_diff_2}
\mathcal{A}(\psi, v)+\left(\tfrac{\partial a}{\partial u}(\cdot,u)\psi,v\right)_{L^2(\Omega)}=-\left(\tfrac{\partial^2 a}{\partial u^2}(\cdot,u)\phi_{w_1}\phi_{w_2}, v \right)_{L^2(\Omega)} \quad \forall v \in \tilde H^s(\Omega),
\end{equation}
where $u = \mathcal{S}z$ and $\phi_{w_i}=\mathcal{S}'(z)w_i$, with $i \in \{1,2\}$.
\end{theorem}
\begin{proof}
The first order Fr\'echet differentiability of $\mathcal{S}$ from $L^{r}(\Omega)$ into $\tilde H^s(\Omega) \cap L^{\infty}(\Omega)$ follows from a slight modification of the proof of \cite[Theorem 4.17]{MR2583281} that basically entails to replace $H^1(\Omega)$ by $\tilde H^s(\Omega)$ and $C(\bar \Omega)$ by $L^{\infty}(\Omega)$. These arguments also show that $\phi=\mathcal{S}'(z)w\in \tilde H^{s}(\Omega)\cap L^{\infty}(\Omega)$ corresponds to the unique solution to \eqref{eq:aux_adjoint}; since $w \in L^r(\Omega)$ and $\frac{\partial a}{ \partial u} (x,u) \geq 0$ for a.e.~$x\in \Omega$ and all $u \in \mathbb{R}$, problem \eqref{eq:aux_adjoint} is well-posed.

The second order Fr\'echet differentiability of $\mathcal{S}$ can be obtained by utilizing the implicit function theorem \cite[Theorem 4.24]{MR2583281}. Let us introduce
the linear map 
\[
\mathfrak{R}: L^r(\Omega) \rightarrow \tilde H^s(\Omega) \cap L^{\infty}(\Omega):
\qquad 
\mathfrak{f} \mapsto \mathfrak{u},
\qquad
\mathcal{A}(\mathfrak{u},v) = \langle \mathfrak{f},v\rangle \quad \forall v \in \tilde H^s(\Omega).
\]
Define $\mathfrak{F}: [\tilde H^s(\Omega) \cap L^{\infty}(\Omega)] \times L^r(\Omega) \rightarrow \tilde H^s(\Omega) \cap L^{\infty}(\Omega)$ by $\mathfrak{F}(u,z) := u - \mathfrak{R}(z - a(\cdot,u))$. We first observe that $\mathfrak{F}$ is of class $C^2$. Second, $\mathfrak{F}(\mathcal{S}z,z) = 0$. Third, since $\partial \mathfrak{F} / \partial u (u,z) v =  v + \mathfrak{R} \partial a/ \partial u (\cdot,u)v$, it can be deduced that the linear map $\partial \mathfrak{F} / \partial u (u,z)$ is invertible from $\tilde H^s(\Omega) \cap L^{\infty}(\Omega)$ into itself. The implicit function theorem thus implies that $\mathcal{S}$ is of class $C^2$. The fact that $\psi$ solves \eqref{eq:aux_adjoint_diff_2} follows from differentiating the relation $\mathfrak{F}(\mathcal{S}z,z) = 0$; see \cite[Theorem 4.24 (ii)]{MR2583281} for details.
\end{proof}

The following result is standard: If $\bar{z}\in \mathbb{Z}_{ad}$ denotes a locally optimal control for problem \eqref{eq:min}--\eqref{eq:weak_st_eq}, then 
$
j'(\bar z) (z - \bar z) \geq 0
$
for all $z \in \mathbb{Z}_{ad}$ \cite[Lemma 4.18]{MR2583281}. To explore this inequality, we define the adjoint state $p \in \tilde H^s(\Omega) \cap L^{\infty}(\Omega)$ as the solution to
\begin{equation}\label{eq:adj_eq}
\mathcal{A}(v,p) + \left(\tfrac{\partial a}{\partial u}(\cdot,u)p,v\right)_{L^2(\Omega)} = \left(\tfrac{\partial L}{\partial u}(\cdot,u),v\right)_{L^2(\Omega)}
\quad \forall v \in \tilde H^s(\Omega).
\end{equation}
Assumption \textnormal{\ref{A2}} guarantees that $\partial a/\partial u (x,u) \geq 0$ for a.e.~$x \in \Omega$ and for all $u \in \mathbb{R}$.  Assumption \textnormal{\ref{B2}} yields $\partial L/ \partial u(\cdot,u) \in L^r(\Omega)$ for $r>n/2s$. The existence of a unique solution $p \in \tilde H^s(\Omega) \cap L^{\infty}(\Omega)$ to problem \eqref{eq:adj_eq} is thus immediate.

We present first order necessary optimality conditions for problem \eqref{eq:min}--\eqref{eq:weak_st_eq}.

\begin{theorem}[first order necessary optimality conditions]
\label{thm:optimality_cond}
Let $n \geq 2$, $s \in (0,1)$, and $r>n/2s$. Assume that \textnormal{\ref{A1}}--\textnormal{\ref{A3}} and \textnormal{\ref{B1}}--\textnormal{\ref{B2}} hold. Then, every locally optimal control $\bar z \in  \mathbb{Z}_{ad}$ satisfies the variational inequality
\begin{equation}\label{eq:var_ineq}
\left( \bar p +\alpha \bar{z},z-\bar{z}\right)_{L^2(\Omega)}\geq 0 \quad \forall z\in \mathbb{Z}_{ad},
\end{equation}
where $\bar p \in \tilde H^s(\Omega) \cap L^{\infty}(\Omega)$ denotes the solution to \eqref{eq:adj_eq} with $u$ replaced by $\bar{u} = \mathcal{S} \bar z$.
\end{theorem}
\begin{proof} 
Define $\ell: L^{\infty}(\Omega) \rightarrow \mathbb{R}$ by $\ell(u) = \int_{\Omega} L(x,u(x)) \mathrm{d}x$ and observe that \textnormal{\ref{B1}}--\textnormal{\ref{B2}} yield the Fr\'echet differentiability of $\ell$ on $L^{\infty}(\Omega)$.
Since $\mathcal{S}$ is differentiable as a map from $L^r(\Omega)$ into $H^s(\Omega) \cap L^{\infty}(\Omega)$, we thus deduce the Fr\'echet differentiability of $j$ as a map from $L^{\sigma}(\Omega)$ to $\mathbb{R}$, where $\sigma = \max \{ n/2s, 2 \}$, upon noticing that $L^2(\Omega) \ni z \mapsto \| z \|^2_{L^2(\Omega)} \in \mathbb{R}$ is also differentiable. Basic computations thus reveal
\begin{equation}
j'(\bar z) h = \int_{\Omega} \left( \tfrac{\partial L}{\partial u} (x, \mathcal{S} \bar z(x) ) \mathcal{S}'(\bar z)h(x)  +  \alpha \bar z(x) h(x) \right) \mathrm{d}x, \quad h \in L^{\sigma}(\Omega).
\label{eq:aux_variational_inequality}
\end{equation}
Set $h = z - \bar z \in \mathbb{Z}_{ad}$ and define $\chi= \mathcal{S}'(\bar z)h$. Setting $v = \chi$ in problem \eqref{eq:adj_eq} and $v = \bar p$ in the problem that $\chi$ solves allow us to obtain $(z - \bar z, \bar p)_{L^2(\Omega)} = (\tfrac{ \partial L}{\partial u}(\cdot,\bar u),\chi)_{L^2(\Omega)}$. Replace this identity into \eqref{eq:aux_variational_inequality} to obtain \eqref{eq:var_ineq}. This concludes the proof.
\end{proof}

Define $\Pi_{[\mathfrak{a},\mathfrak{b}]} : L^1(\Omega) \rightarrow  \mathbb{Z}_{ad}$ by $\Pi_{[\mathfrak{a},\mathfrak{b}]}(v) := \min\{ \mathfrak{b}, \max\{ v, \mathfrak{a} \} \}$ a.e. in  $\Omega$. We present the following projection formula: If $\bar{z} \in \mathbb{Z}_{ad}$ denotes a locally optimal control for problem \eqref{eq:min}--\eqref{eq:weak_st_eq}, then
 \cite[Section 4.6]{MR2583281}
\begin{equation}\label{eq:projection_control} 
\bar{z}(x):=\Pi_{[\mathfrak{a},\mathfrak{b}]}(-\alpha^{-1}\bar{p}(x)) \textrm{ a.e.}~x \in \Omega.
\end{equation}
Since $\bar p \in \tilde H^s(\Omega) \cap L^{\infty}(\Omega)$ and $s \in (0,1)$, it is immediate that $\bar z \in H^{s}(\Omega) \cap L^{\infty}(\Omega)$; further regularity properties for $\bar z$ are obtained in Theorem \ref{thm:regularity_space} below.

\subsection{Second order optimality conditions}\label{sec:2nd_order}
In Theorem \ref{thm:optimality_cond} we derived a first order \emph{necessary} optimality condition. Since our optimal control problem is not convex, sufficiency requires the use of second order optimality conditions. The purpose of this section is thus to derive second order \emph{necessary} and \emph{sufficient} optimality conditions. To accomplish this task, we begin by introducing some preliminary concepts. Let $\bar{z} \in \mathbb{Z}_{ad}$ satisfies \eqref{eq:var_ineq}. Define $\bar{\mathfrak{p}} :=  \bar p + \alpha \bar z$. Observe that \eqref{eq:var_ineq} immediately yields
\begin{equation}\label{eq:derivative_j}
\bar{\mathfrak{p}}(x) 
\begin{cases}
= 0  & \text{ a.e.}~x \in \Omega \text{ if } \texttt{a}< \bar{z} < \texttt{b}, \\
\geq  0  & \text{ a.e.}~x \in \Omega \text{ if }\bar{z}=\texttt{a}, \\
\leq  0  & \text{ a.e.}~x \in \Omega \text{ if } \bar{z}=\texttt{b}.
\end{cases}
\end{equation}
Define the \emph{cone of critial directions}
$
C_{\bar{z}}:=\{v\in L^2(\Omega): \eqref{eq:sign_cond} \text{ holds and } \bar{\mathfrak{p}}(x) \neq 0 \implies v(x) = 0 \},
$
where condition \eqref{eq:sign_cond} reads as follows:
\begin{equation}
\label{eq:sign_cond}
v(x)
\geq 0  \text{ a.e.}~x\in\Omega \text{ if } \bar{z}(x)=\mathfrak{a},
\qquad
v(x) \leq 0 \text{ a.e.}~x\in\Omega \text{ if } \bar{z}(x)=\mathfrak{b}.
\end{equation}

The following result is instrumental.
\begin{proposition}[$j$ is of class $C^2$]
\label{pro:diff_properties_j}
Let $n \geq 2$, $s \in (0,1)$, $r>n/2s$, and $\sigma = \max \{ 2, n/2s \}$. Assume that \textnormal{\ref{A1}}--\textnormal{\ref{A3}} and \textnormal{\ref{B1}}--\textnormal{\ref{B2}} hold. Then the reduced cost $j: L^{\sigma}(\Omega) \rightarrow \mathbb{R}$ is of class $C^2$. In addition, for $z,w_1,w_2 \in L^{\sigma}(\Omega)$, we have
\begin{equation}\label{eq:charac_j2}
j''(z)(w_1,w_2)
=
\int_{\Omega}
\left(
\frac{\partial^2 L}{\partial u^2}(x,u)
\phi_{w_1}\phi_{w_2}
+ 
\alpha w_1 w_2 
-
p \frac{\partial^2 a}{\partial u^2}(x,u)\phi_{w_1}\phi_{w_2} 
\right)
\mathrm{d}x,
\end{equation}
where $u = \mathcal{S}z$, $p$ solves \eqref{eq:adj_eq} and $\phi_{w_i}=\mathcal{S}'(z)w_i$, with $i \in \{1,2 \}$.
\end{proposition}
\begin{proof}
The fact that $j$ is first order differentiable follows from Theorem \ref{thm:optimality_cond}. Theorem \ref{thm:properties_C_to_S} guarantees that $\mathcal{S}$ is second order Fr\'echet differentiable as a map from $L^{r}(\Omega)$ into $\tilde H^s(\Omega) \cap L^{\infty}(\Omega)$. In view of \textnormal{\ref{B1}}--\textnormal{\ref{B2}}, the map $u \mapsto \ell(u) := \int_{\Omega}L(x,u(x)) \mathrm{d}x$ is second order Fr\'echet differentiable as well as a map from $L^{\infty}(\Omega)$
to $\mathbb{R}$. The chain rule allows us to conclude that $j \in C^2$. The identity \eqref{eq:charac_j2} follows from the arguments elaborated in \cite[Section 4.10]{MR2583281}.
\end{proof}

We are now in position to formulate second order \emph{necessary} optimality conditions.

\begin{theorem}[second order necessary optimality conditions]
\label{thm:nec_opt_cond}
Let $n \in \{ 2,3 \}$ and $s > n/4$. If $\bar{z} \in \mathbb{Z}_{ad}$ denotes a locally minimum for problem \eqref{eq:min}--\eqref{eq:weak_st_eq}, then 
\begin{equation}\label{eq:second_order_nec}
j''(\bar{z})v^2 \geq 0 \quad \forall v\in C_{\bar{z}},
\end{equation}
where $C_{\bar{z}}:=\{v\in L^2(\Omega): \eqref{eq:sign_cond} \text{ holds and } \bar{\mathfrak{p}}(x) \neq 0 \implies v(x) = 0 \}$.
\end{theorem}
\begin{proof}
Let $v\in C_{\bar{z}}$. Define, for every $k \in \mathbb{N}$ and for a.e.~$x\in \Omega$, the function
\begin{equation*}
v_{k}(x):=
\begin{cases}
\qquad \quad 0 \quad &\text{ if }\quad x: \mathfrak{a} < \bar{z}(x) < \mathfrak{a}+ \tfrac{1}{k}, 
\quad 
 \mathfrak{b}-\tfrac{1}{k} < \bar{z}(x) <  \mathfrak{b}, 
\\
\Pi_{[-k,k]}(v(x)) &\text{ otherwise}. 
\end{cases}
\end{equation*}
Since $v \in C_{\bar z}$, we deduce that $v_k \in C_{\bar z}$. In addition, $|v_{k}(x)|\leq |v(x)|$ and $v_{k}(x)\rightarrow v(x)$ for a.e.~$x \in \Omega$ as $k \uparrow \infty$; therefore $v_k \rightarrow v$ in $L^2(\Omega)$. Now, since $\bar z + \rho v_k \in \mathbb{Z}_{ad}$, for $\rho \in (0,k^{-2}]$, and $\bar z$ is locally optimal for $j$ we arrive, for $\rho$ sufficiently small, at
\begin{equation}
0 \leq \tfrac{1}{\rho}[ j(\bar z + \rho v_k) - j(\bar z) ] = j'(\bar z) v_k + \tfrac{\rho}{2} j''(\bar z + \theta_k \rho v_k) v_k^2, \quad \theta_k \in (0,1).
\label{eq:aux_second_order}
\end{equation}
Observe that \eqref{eq:aux_variational_inequality} and $v_k \in C_{\bar z}$ reveal that $j'(\bar z) v_k = \int_{\Omega} \bar{\mathfrak{p}}(x) v_k(x) \mathrm{d}x = 0$. We thus divide by $\rho$ in \eqref{eq:aux_second_order}, utilize \eqref{eq:charac_j2}, and let $\rho \downarrow 0$ to obtain $j''(\bar z) v_k^2 \geq 0$.  Let $k \uparrow \infty$ and invoke \eqref{eq:charac_j2}, again, and $\| v_k - v\|_{L^2(\Omega)} \rightarrow 0$ to conclude.
\end{proof}

We now provide a \emph{sufficient} second order optimality condition with a minimal gap with respect to the \emph{necessary} one derived in Theorem \ref{thm:nec_opt_cond}.

\begin{theorem}[second order sufficient optimality conditions]
\label{thm:suff_opt_cond}
Let $n \in \{ 2,3 \}$ and $s > n/4$. Let $\bar u \in \tilde H^s(\Omega)$, $\bar p \in \tilde H^s(\Omega)$, and $\bar{z} \in \mathbb{Z}_{ad}$ satisfy the first order optimality conditions \eqref{eq:weak_st_eq}, \eqref{eq:adj_eq}, and \eqref{eq:var_ineq}. If
\begin{equation}\label{eq:second_order_suff}
j''(\bar{z})v^2 > 0 \quad \forall v\in C_{\bar{z}} \setminus \{ 0 \},
\end{equation}
then there exists $\kappa > 0$ and $\mu >0$ such that
\begin{equation}
\label{eq:quadratic_growing}
j( z) \geq j(\bar z) + \tfrac{\kappa}{2} \| z - \bar  z\|_{L^2(\Omega)}^2 
\end{equation}
for every $z \in \mathbb{Z}_{ad}$ such that $\| \bar z - z \|_{L^2(\Omega)} \leq \mu$.
\end{theorem}
\begin{proof}
We proceed by contradiction and assume that for every $k \in \mathbb{N}$ there exists an element $z_k \in \mathbb{Z}_{ad}$ such that
\begin{equation}
\label{eq:contradiction_argument}
\| \bar z - z _{k} \|_{L^2(\Omega)} < \tfrac{1}{k},
\quad
j( z_k) < j(\bar z) + \tfrac{1}{2k} \| \bar  z - z_k\|_{L^2(\Omega)}^2.
\end{equation}
Define $\rho_k:= \|  z _{k} - \bar z  \|_{L^2(\Omega)}$ and $v_k:= \rho_k^{-1} ( z_k - \bar z)$. Notice that there exists a nonrelabeled subsequence $\{ v_k \}_{k=1}^{\infty} \subset L^2(\Omega)$ such that $v_k \rightharpoonup v$ in $L^2(\Omega)$ as $k \uparrow \infty$. 

We now proceed on the basis of three steps:

\noindent \framebox{1} We prove that $v \in C_{\bar z}$. Since the set of elements satisfying \eqref{eq:sign_cond} is closed and convex in $L^2(\Omega)$ and, for every $k \in \mathbb{N}$, $v_k$ belongs to this set, we deduce that $v$ satisfies \eqref{eq:sign_cond}. It suffices to prove that $\bar{\mathfrak{p}}(x) \neq 0$ implies $v(x) = 0$. In view of \eqref{eq:var_ineq}, we  deduce that
$\int_{\Omega} \bar{\mathfrak{p}}(x) v(x) \mathrm{d}x \geq 0$ because 
$
\int_{\Omega} \bar{\mathfrak{p}}(x) v_k(x) \mathrm{d}x 
=
\rho_k^{-1} \int_{\Omega} \bar{\mathfrak{p}}(x)(z_k(x) - \bar z(x) )\mathrm{d}x \geq 0$. On the other hand, observe that \eqref{eq:contradiction_argument} and the mean value theorem reveal that
\[
j(z_k) - j(\bar z) = j'(\bar z + \theta_k( z_k - \bar z )) (z_k - \bar z) < \tfrac{1}{2k} \| \bar  z - z_k\|_{L^2(\Omega)}^2 = \tfrac{\rho_k^2}{2k}, 
\quad
\theta_k \in (0,1).
\]
Divide by $\rho_k$ and let $k \uparrow \infty$ to arrive at $ j'(\bar z + \theta_k( z_k - \bar z )) v_k < (2k)^{-1}\rho_k  \rightarrow 0$ as $k \uparrow \infty$. Define $\hat z_k := \bar z + \theta_k( z_k - \bar z )$. Since $s>n/4$ and $\hat z_k \rightarrow \bar z$ in $L^2(\Omega)$, as $k \uparrow \infty$, we have
\[
\hat u_k:= \mathcal{S}(\hat z_k) \rightarrow \mathcal{S}(\bar z) = \bar u 
\textrm{ in } \tilde H^s(\Omega) \cap L^{\infty}(\Omega),
\quad
\tfrac{\partial L}{\partial u} (\cdot ,\hat u_k) \rightarrow \tfrac{\partial L}{\partial u} (\cdot, \bar u) 
\textrm{ in }
L^{r}(\Omega),
\]
upon invoking \textnormal{\ref{B2}}. Consequently, $\hat p_k \rightarrow \bar p$ in $\tilde H^s(\Omega) \cap L^{\infty}(\Omega)$ as $k \uparrow \infty$. Here, $\hat p_k$ denotes the solution to \eqref{eq:adj_eq} with 
$u$ replaced by $\hat u_k$. Thus,
\[
\int_{\Omega} \bar{\mathfrak{p}}(x) v(x) \mathrm{d}x  =
\lim_{k \uparrow \infty} 
\int_{\Omega} \left[ \hat p_k (x)  + \alpha \hat{z_k}(x) \right]v_k(x)\mathrm{d}x
=
\lim_{k \uparrow \infty} 
 j'( \bar z + \theta_k( z_k - \bar z ) )v_k \leq 0.
\]
We have thus deduced that $\int_{\Omega} \bar{\mathfrak{p}}(x) v(x) \mathrm{d}x = \int_{\Omega} |\bar{\mathfrak{p}}(x) v(x)| \mathrm{d}x = 0$. Consequently, $\bar{\mathfrak{p}}(x) \neq 0$ implies $v(x) = 0$ for a.e.~$x\in \Omega$. This proves that $v \in C_{\bar z}$.

\noindent \framebox{2} We prove that $v=0$. We begin with an application of Taylor's theorem and write
\[
j(z_k) = j(\bar z) + \rho_k j'(\bar z) v_k + \tfrac{\rho_k^2}{2}j''(\hat z_k) v_k^2, \quad
\theta_k \in (0,1),
\]
where $\hat z_k = \bar z + \theta_k(z_k - \bar z)$ and $\rho_kv_k = z_k - \bar z$. Now, $j'(\bar z) v_k \geq 0$ and \eqref{eq:contradiction_argument} yield 
\[
\tfrac{\rho_k^2}{2}j''(\hat z_k) v_k^2 \leq j(z_k) - j(\bar z) < \tfrac{1}{2k} \| \bar  z - z_k\|_{L^2(\Omega)}^2.
\]
This implies that $j''(\hat z_k) v_k^2 < k^{-1}$. Consequently, $j''(\hat z_k) v_k^2 < k^{-1} \rightarrow 0$ as $k \uparrow \infty$.

We now prove that $j''(\bar z) v^2 \leq \liminf_{k} j''(\hat z_k) v_k^2$. We begin by noticing that
\begin{equation*}
j''(\hat z_k)v_k^2
=
\int_{\Omega}
\left(
\frac{\partial^2 L}{\partial u^2}(x, \hat u_k)
\phi_{v_k}^2
-
\hat p_k \frac{\partial^2 a}{\partial u^2}(x,\hat u_k)\phi_{v_k}^2
+ 
\alpha v_k^2 
\right)
\mathrm{d}x.
\end{equation*}
As $k \uparrow \infty$, $\hat z_k \rightarrow \bar z$ and $v_k \rightharpoonup v$ in $L^2(\Omega)$. We thus have
$
\hat u_k \rightarrow \bar u
$
and
$
\hat p_k \rightarrow \bar p 
$
in $\tilde H^s(\Omega) \cap L^{\infty}(\Omega)$ and $\phi_{v_k} \rightharpoonup \phi_{v}$ in $\tilde H^s(\Omega)$; the latter implies that $\phi_{v_k} \rightarrow \phi_{v}$ in $L^{\mathfrak{q}}(\Omega)$, as $k \uparrow \infty$, for $\mathfrak{q} < 2n/(n-2s)$. Invoke \textnormal{\ref{B2}} to obtain 
\begin{multline*}
\left| \int_{\Omega}
\left(
\tfrac{\partial^2 L}{\partial u^2}(x, \hat u_k)
\phi_{v_k}^2
-
\tfrac{\partial^2 L}{\partial u^2}(x, \bar u)
\phi_{v}^2
\right)
\mathrm{d}x\right|
\leq
\| \phi_{v_k} \|^2_{L^{\mathfrak{q}}(\Omega)} \left \|  \tfrac{\partial^2 L}{\partial u^2}(\cdot, \hat u_k) - \tfrac{\partial^2 L}{\partial u^2}(\cdot, \bar u)\right\|_{L^r(\Omega)}
\\
+ \| \phi_{\mathfrak{m}} \|_{L^r(\Omega)}   \| \phi_v + \phi_{v_k} \|_{L^{\mathfrak{q}}(\Omega)}  \| \phi_v - \phi_{v_k} \|_{L^{\mathfrak{q}}(\Omega)} \rightarrow 0, \quad k \uparrow \infty.
\end{multline*}
On the other hand, invoke \textnormal{\ref{A2}} to derive
\begin{multline*}
\left| \int_{\Omega}
\left(
\bar p \tfrac{\partial^2 a}{\partial u^2}(x, \bar u )\phi_{v}^2
-
\hat p_k \tfrac{\partial^2 a}{\partial u^2}(x,\hat u_k)\phi_{v_k}^2
\right)
\mathrm{d}x\right|
\leq C_{\mathfrak{m}} \| \phi_v  \|^2_{L^{\EO{\mathfrak{q}}}(\Omega)} \| \bar p - \hat p_k \|_{L^{\EO{r}}(\Omega)} 
\\
+ 
C_{\mathfrak{m}} \| \hat p_k \|_{L^{r}(\Omega)}  \left( \| \bar u - \hat u_k \|_{L^{\infty}(\Omega)}  \| \phi_v  \|^2_{L^{\mathfrak{q}}(\Omega)}
+
\| \phi_v + \phi_{v_k}  \|_{L^{\mathfrak{q}}(\Omega)}
\| \phi_v - \phi_{v_k}  \|_{L^{\mathfrak{q}}(\Omega)}
\right) \rightarrow 0
\end{multline*}
as $k \uparrow \infty$. Finally, observe that $\| v \|^2_{L^2(\Omega)}  \leq \liminf_{k \uparrow \infty} \| v_k\|^2_{L^2(\Omega)}$ because $\| \cdot \|^2_{L^2(\Omega)}$ is weakly lower semicontinuous. We thus conclude that $j''(\bar z) v^2 \leq \liminf_{k \uparrow \infty} j''(\hat z_k) v_k^2$. 

Finally, since $ \liminf_{k \uparrow \infty} j''(\hat z_k) v_k^2\leq 0$ and $v \in C_{\bar z}$, \eqref{eq:second_order_suff} implies that $v = 0$.

\noindent \framebox{3} Since $v=0$, we have that $\phi_{v_k} \rightarrow 0$ in $L^{\mathfrak{q}}(\Omega)$, for $\mathfrak{q} < 2n/(n-2s)$, as $k \uparrow \infty$. Consequently, from the identity
\[
\alpha = \alpha \| v_k\|_{L^2(\Omega)}^2 =  j''(\hat z_k) v_k^2 
-
\int_{\Omega}
\left(
\frac{\partial^2 L}{\partial u^2}(x, \hat u_k)
\phi_{v_k}^2
-
\hat p_k \frac{\partial^2 a}{\partial u^2}(x,\hat u_k)\phi_{v_k}^2
\right)
\mathrm{d}x
\] 
and the fact that $\liminf_{k \uparrow \infty} j''(\hat z_k) v_k^2 \leq 0$, we conclude that $\alpha \leq 0$.
This contradicts the fact that $\alpha > 0$ and concludes the proof.
\end{proof}

Define, for $\tau >0$,
\begin{equation}
C_{\bar{z}}^\tau:=\{v\in L^2(\Omega): \eqref{eq:sign_cond} \textnormal{ holds and } |\bar{\mathfrak{p}}(x)|>\tau \implies v(x)=0 \}.
\label{def:critical_cone_tau}
\end{equation}

\begin{theorem}[equivalent optimality conditions]
\label{thm:equivalent_opt_cond}
Let $n \in \{ 2,3 \}$ and $s > n/4$. Let $\bar{u} \in \tilde H^s(\Omega)$, $\bar p \in \tilde H^s(\Omega)$, and $\bar z \in \mathbb{Z}_{ad}$ satisfy the first order optimality conditions \eqref{eq:weak_st_eq}, \eqref{eq:adj_eq}, and \eqref{eq:var_ineq}. Thus, \eqref{eq:second_order_suff} is equivalent to
\begin{equation}\label{eq:second_order_equivalent}
\exists \mu, \tau >0: \quad j''(\bar{z})v^2 \geq \mu \|v\|_{L^2(\Omega)}^2 \quad \forall v \in C_{\bar{z}}^\tau,
\end{equation}
where $C_{\bar{z}}^\tau$ is defined in \eqref{def:critical_cone_tau}.
\end{theorem}
\begin{proof}
Since $C_{\bar z} \subset C_{\bar z}^{\tau}$, we immediately conclude that \eqref{eq:second_order_equivalent} implies \eqref{eq:second_order_suff}.  To prove that \eqref{eq:second_order_suff} implies \eqref{eq:second_order_equivalent} we proceed by contradiction. Assume that, for $\tau >0$,
\[
\exists v_{\tau} \in C_{\bar z}^{\tau}:
\quad
j''(\bar z) v_{\tau}^2 < \tau \| v_{\tau}\|_{L^2(\Omega)}^2.
\]
Define $w_{\tau} := \|v_\tau \|_{L^2(\Omega)}^{-1}v_{\tau} $. Note that, up to a nonrelabeled subsequence if necessary,
\begin{equation}
\label{eq:weak_seq}
w_{\tau} \in C_{\bar z}^{\tau},
\qquad
\|w_{\tau}\|_{L^2(\Omega)}=1, \qquad j''(\bar{z})w_{\tau}^2<\tau, \qquad w_{\tau} \rightharpoonup w  \text{ in } L^2(\Omega).
\end{equation}

We prove that $w \in C_{\bar z}$. Since the set of elements satisfying \eqref{eq:sign_cond} is weakly closed in $L^2(\Omega)$, we conclude that $w$ satisfies \eqref{eq:sign_cond} as well. On the other hand, 
\[
\int_{\Omega} \bar{\mathfrak{p}}(x) w(x) \mathrm{d}x = \lim_{\tau \downarrow 0} \int_{\Omega} \bar{\mathfrak{p}}(x) w_{\tau}(x) \mathrm{d}x = \lim_{\tau \downarrow 0} \int_{|\bar{\mathfrak{p}}(x)| \leq \tau} \bar{\mathfrak{p}}(x) w_{\tau}(x) \mathrm{d}x 
\leq 
\lim_{\tau \downarrow 0} \tau \sqrt{|\Omega|} = 0,
\]
where we have used that $\bar{\mathfrak{p}} \in L^2(\Omega)$, $w_{\tau}\rightharpoonup w$ in $L^2(\Omega)$, $w_{\tau} \in C_{\bar z}^{\tau}$, and $\| w_{\tau}\|_{L^2(\Omega)} = 1$. As a result, $\int_{\Omega} |\mathfrak{p}(x) w(x)| \mathrm{d}x = \int_{\Omega}\mathfrak{p}(x) w(x) \mathrm{d}x = 0$. This implies that if $|\mathfrak{p}(x)| \neq 0$, then $w(x) = 0$ for a.e.~$x \in \Omega$. We can thus conclude that $w \in C_{\bar z}$.

We now prove that $w = 0$. Since $w \in C_{\bar z}$, \eqref{eq:second_order_suff} implies that either $w = 0$ or $j''(\bar z) w^2 >0$. On the other hand, the arguments elaborated in the step 2 of the proof of Theorem \ref{thm:suff_opt_cond} in conjunction with \eqref{eq:weak_seq} yield
$
j''(\bar z) w^2 
\leq 
\liminf_{\tau \downarrow 0} j''(\bar z) w_{\tau}^2
\leq
\limsup_{\tau \downarrow 0} j''(\bar z) w_{\tau}^2
\leq 0.
$
Consequently, $w = 0$ and $\lim_{\tau \downarrow 0} j''(\bar z) w_{\tau}^2 = 0$.

Finally, since $w = 0$ and $w_{\tau} \rightharpoonup 0$ in $L^2(\Omega)$ as $\tau \downarrow 0$, we have that $\phi_{w_{\tau}} \rightarrow 0$ in $L^{\mathfrak{q}}(\Omega)$, as $\tau \downarrow 0$, for $\mathfrak{q} < 2n/(n-2s)$. Thus,
$
\alpha = \alpha \| w_{\tau} \|_{L^2(\Omega)}^2 \leq  \liminf_{\tau \downarrow 0} j''(\bar z) w_{\tau}^2 = 0,
$
which is a contradiction. This concludes the proof.
\end{proof}

\subsection{Regularity estimates}
\label{sec:regularity}
In this section, we derive regularity estimates for the optimal control variables. To accomplish this task, the following regularity result for the linear case $a \equiv 0$ will be of importance.

\begin{proposition}[Sobolev regularity on smooth domains]
Let $n \geq 1$, $s \in (0,1)$, and $\Omega$ be a domain such that $\partial \Omega \in C^{\infty}$. Let $\mathsf{u}$ be the solution to $(-\Delta)^s \mathsf{u} = \mathsf{f}$ in $\Omega$ and $\mathsf{u} = 0$ in $\Omega^c$. If $\mathsf{f} \in H^{t}(\Omega)$, for some $t \geq -s$, then $\mathsf{u} \in H^{s + \vartheta}(\Omega)$, where $\vartheta = \min \{s+t,1/2-\epsilon \}$ and $\epsilon > 0$ is arbitrarily small. In addition, we have 
\begin{equation}
\| \mathsf{u} \|_{H^{s+\vartheta}(\Omega)} \lesssim \| \mathsf{f} \|_{H^{t}(\Omega)},
\label{eq:regularity_state_smooth}
 \end{equation}
where the hidden constant depends on $\Omega$, $n$, $s$, and $\vartheta$.
\label{pro:state_regularity_smooth}
\end{proposition}
\begin{proof}
See \cite{MR3276603,MR0185273}.
\end{proof}

Observe that smoothness of $\mathsf{f}$ does not ensure that the solution to $(-\Delta)^s \mathsf{u} = \mathsf{f}$ in $\Omega$ and $\mathsf{u} = 0$ in $\Omega^c$ is any smoother than $\cap \{ H^{s+1/2-\epsilon}(\Omega): \epsilon > 0 \}$.

To present regularity estimates, we will assume that, in addition to \ref{A1}--\ref{A3} and \ref{B1}--\ref{B2}, the nonlinear functions $a$ and $L$ satisfy \ref{C1}--\ref{C2}.

\begin{theorem}[regularity estimates: $s \in (0,1)$]
Let $n \geq 2$ and $s \in (0,1)$. If $\Omega$ is such that $\partial \Omega \in C^{\infty}$, then
$\bar u, \bar p, \bar{z} \in H^{s + 1/2 - \epsilon}(\Omega)$, where $\epsilon$ denotes an arbitrarily small positive constant.
\label{thm:regularity_space}
\end{theorem}
\begin{proof}
Since $\bar z \in \mathbb{Z}_{ad}$ and $a(\cdot,0) \in L^2(\Omega)$, we apply Proposition \ref{pro:state_regularity_smooth} with $t = 0$ to obtain $\bar u \in H^{s + \nu}(\Omega)$, where $\nu = \min \{ s , 1/2 - \epsilon \}$ and $\epsilon >0$ is arbitrarily small, and  
\begin{equation}
\| \bar u \|_{H^{s + \nu}(\Omega)} \lesssim \| \bar z - a(\cdot, \bar u)\|_{L^2(\Omega)} \lesssim \| \bar z \|_{L^{2}(\Omega)} + \| a(\cdot,0) \|_{L^2(\Omega)},
\label{eq:first_estimate_u}
\end{equation}
upon utilizing that $a$ is locally Lipschitz in the second variable and $\| \bar u \|_s \lesssim \| \bar z \|_{H^{-s}(\Omega)}$. We now obtain a first regularity estimate for $\bar p$. To accomplish this task, we invoke Proposition \ref{pro:state_regularity_smooth} to obtain $\bar p \in H^{s + \iota}(\Omega)$, where $\iota = \min \{ s + \lambda ,  \frac{1}{2} - \epsilon \}$, $\lambda = \min \{ 0, \frac{1}{2} - s - \epsilon \}$, and $\epsilon >0$ is arbitrarily small. In addition, we have the following estimate:
\begin{multline}
\| \bar p  \|_{H^{s + \iota}(\Omega)} \lesssim \left \| \tfrac{\partial a }{\partial u} (\cdot,\bar u) \bar p \right \|_{L^{2}(\Omega)} + \left \| \tfrac{\partial L}{\partial u}(\cdot,\bar u) \right \|_{H^{\frac{1}{2} - s - \epsilon}(\Omega)}
\\
\lesssim  \| \bar p \|_{L^{\infty}(\Omega)}
\left \| \tfrac{\partial a }{\partial u} (\cdot,\bar u) \right \|_{L^2(\Omega)}
 + \left \|  \tfrac{\partial L}{\partial u} (\cdot,\bar u) \right \|_{H^{\frac{1}{2} - s - \epsilon}(\Omega)}.
\label{eq:first_estimate_p}
\end{multline}
In view of \eqref{eq:projection_control}, \cite[Theorem 1]{MR1173747}
yields $\bar z \in H^{s + \iota}(\Omega)$ with a similar estimate. 

We now consider three cases.

\noindent 
\boxed{1} $s \in ( \frac{1}{2}, 1 )$: Observe that $\nu = \iota = \tfrac{1}{2} - \epsilon$. Thus, $\bar u , \bar p, \bar z  \in H^{s + \frac{1}{2} - \epsilon}(\Omega)$ for $\epsilon >0$ being arbitrarily small.
In addition, the estimates \eqref{eq:first_estimate_u} and \eqref{eq:first_estimate_p} yield
\begin{multline*}
\| \bar u \|_{H^{s+\frac12 - \epsilon}(\Omega)} 
+
\| \bar p \|_{H^{s+\frac12 - \epsilon}(\Omega)} 
+
\| \bar z \|_{H^{s + \frac{1}{2} - \epsilon}(\Omega)}
\lesssim 
 \| \bar p \|_{L^{\infty}(\Omega)}
\left \| \tfrac{\partial a }{\partial u} (\cdot,\bar u) \right \|_{L^2(\Omega)}
\\
+
\| \bar z \|_{L^{2}(\Omega)}
+
\|  \tfrac{\partial L}{\partial u} (\cdot,\bar u) \|_{H^{\frac{1}{2} - s - \epsilon}(\Omega)}
+ \| a(\cdot,0) \|_{L^2(\Omega)}
=: \mathfrak{B}.
\end{multline*}

\noindent  
\boxed{2} $s = \tfrac{1}{2}$. The proof follows similar arguments.  For brevity, we skip the details.

\noindent  
\boxed{3} $s \in (0,\tfrac{1}{2})$. Here, $\nu = \iota = s$. Thus, $\bar u, \bar p \in H^{2s}(\Omega)$. In view of \eqref{eq:projection_control}, a nonlinear interpolation result based on \cite[Theorem A.1]{MR1786735} and \cite[Lemma 28.1]{Tartar} yields $\bar z \in H^{2s}(\Omega)$. In addition, we have the estimate
\begin{equation*}
\| \bar u \|_{H^{2s}(\Omega)} 
+
\| \bar p \|_{H^{2s}(\Omega)} 
+
\| \bar z \|_{H^{2s}(\Omega)}
\lesssim 
\mathfrak{B}.
\end{equation*}

In what follows, we proceed on the basis of a bootstrap argument as in \cite{MR3429730,MR3990191}.

\noindent  
\boxed{3.1} $s \in [\tfrac{1}{4},\tfrac{1}{2})$. Invoke Proposition \ref{pro:state_regularity_smooth} with $t = 1/2 - s - \epsilon$ to obtain $\bar u \in H^{s+1/2-\epsilon}(\Omega)$, where $\epsilon >0$ is arbitrarily small, and the estimate
\begin{multline*}
\| \bar u \|_{H^{s +  \frac{1}{2} - \epsilon}(\Omega)} 
\lesssim
\| \bar z \|_{H^{\frac{1}{2} -s-\epsilon}(\Omega)}
+
\| a(\cdot,\bar u)  - a(\cdot,0)  \|_{H^{\frac{1}{2} -s-\epsilon}(\Omega)}
+  
\| a(\cdot,0)  \|_{H^{\frac{1}{2} -s-\epsilon}(\Omega)}
\\
\lesssim 
 \| \bar z \|_{H^{2s}(\Omega)} 
 + 
 \| \bar u \|_{H^{2s}(\Omega)}
+
\| a(\cdot,0) \|_{H^{\frac{1}{2} -s-\epsilon}(\Omega)} 
\lesssim \mathfrak{B} + 
\| a(\cdot,0) \|_{H^{\frac{1}{2} -s-\epsilon}(\Omega)}.
\end{multline*}
Observe that $2s > \tfrac{1}{2} - s - \epsilon$ for $\epsilon >0$ being arbitrarily small. On the other hand, in view of assumption \ref{C1}, we have that $\frac{\partial a}{\partial u}(\cdot,\bar u) \in H^s(\Omega)$. Since $\bar p \in H^s(\Omega)$, it thus follows that $\tfrac{\partial a}{\partial u}(\cdot,\bar u) \bar p \in H^s(\Omega)$. In fact, notice that, for $x,y \in \Omega$, we have the estimate
\begin{multline*}
\left|
\tfrac{\partial a }{\partial u} (x,\bar u(x)) \bar p(x) - \tfrac{\partial a }{\partial u} (y,\bar u(y)) \bar p(y)
\right| 
\\ 
\leq
\left| \tfrac{\partial a }{\partial u} (x,\bar u(x)) \right | |\bar p(x) - \bar  p(y)| + |\bar p(y)| \left| \tfrac{\partial a }{\partial u} (x,\bar u(x))  -\tfrac{\partial a }{\partial u} (y,\bar u(y)) \right |.
\end{multline*}
The definition of $| \cdot |_{H^s(\Omega)}$ implies $| \tfrac{\partial a }{\partial u} (\cdot,\bar u) \bar p  |_{H^s(\Omega)} \lesssim  |\bar p|_{H^s(\Omega)} + \| \bar p \|_{L^{\infty}(\Omega)} 
| \tfrac{\partial a }{\partial u} (\cdot,\bar u) |_{H^s(\Omega)},
$
upon utilizing \ref{A3}. We thus invoke Proposition \ref{pro:state_regularity_smooth} with $t = 1/2-s-\epsilon$ to obtain
\[
\| \bar p \|_{H^{s + \frac{1}{2}  -\epsilon}(\Omega)} \lesssim  \|\bar p \|_{H^{s}(\Omega)} + \| \bar p \|_{L^{\infty}(\Omega)} \left \| \tfrac{\partial a}{\partial u} (\cdot,\bar u) \right\|_{H^{s}(\Omega)} + \left \|  \tfrac{\partial L}{\partial u} (\cdot,\bar u) \right \|_{H^{ \frac{1}{2} - s - \epsilon}(\Omega)}.
\]
A nonlinear interpolation argument yields $\bar z \in H^{s + \frac{1}{2}  - \epsilon}(\Omega)$ with a similar estimate.

\noindent  
\boxed{3.2} $s \in [\frac{1}{6},\tfrac{1}{4})$. 
Proposition \ref{pro:state_regularity_smooth} with $t = 1/2 - s - \epsilon$ yields $\bar u \in H^{s+1/2 - \epsilon}(\Omega)$ for $\epsilon >0$ being arbitrarily small. On the other hand, \ref{C1} guarantees that $\tfrac{\partial a }{\partial u} (\cdot,\bar u) \in H^{2s}(\Omega)$. Since $\bar p \in H^{2s}(\Omega)$, we conclude that $\tfrac{\partial a }{\partial u} (\cdot,\bar u) \bar p \in H^{2s}(\Omega)$. Observe that $2s> \frac{1}{2} - s - \epsilon$ and invoke Proposition \ref{pro:state_regularity_smooth} with $t = 1/2-s-\epsilon$ to obtain the estimate
\[
\| \bar p \|_{H^{s + \frac{1}{2}-\epsilon }(\Omega)} \lesssim  \|\bar p \|_{H^{2s}(\Omega)} + \| \bar p \|_{L^{\infty}(\Omega)} \left \| \tfrac{\partial a}{\partial u} (\cdot,\bar u) \right\|_{H^{2s}(\Omega)} + \left \|  \tfrac{\partial L}{\partial u} (\cdot,\bar u) \right \|_{H^{ \frac{1}{2} - s - \epsilon}(\Omega)}.
\]
This implies that $\bar z \in H^{s + \frac{1}{2} - \epsilon}(\Omega)$ with a similar estimate; $\epsilon >0$ is arbitrarily small.

\noindent  
\boxed{3.3} $s \in (0,\tfrac{1}{6})$. Invoke Proposition \ref{pro:state_regularity_smooth} with $t = 2s$ to obtain $\bar u \in H^{4s}(\Omega)$ and
\begin{multline*}
\| \bar u \|_{H^{4s}(\Omega)} 
\lesssim
 \| \bar z \|_{H^{2s}(\Omega)} 
 + 
 \| \bar u \|_{H^{2s}(\Omega)}
+
\| a(\cdot,0) \|_{H^{\frac{1}{2} -s-\epsilon}(\Omega)} 
\lesssim \mathfrak{B} + 
\| a(\cdot,0) \|_{H^{\frac{1}{2} -s-\epsilon}(\Omega)}.
\end{multline*}
On the other hand,  $\tfrac{\partial a }{\partial u} (\cdot,\bar u) \in H^{3s}(\Omega)$. Since $\bar p \in H^{2s}(\Omega)$, we can thus conclude that $\tfrac{\partial a }{\partial u} (\cdot,\bar u) \bar p \in H^{2s}(\Omega)$. Invoke Proposition \ref{pro:state_regularity_smooth} with $t = 2s$ to obtain
\[
\| \bar p \|_{H^{4s }(\Omega)} \lesssim  \|\bar p \|_{H^{2s}(\Omega)} + \| \bar p \|_{L^{\infty}(\Omega)} \left \| \tfrac{\partial a}{\partial u} (\cdot,\bar u) \right\|_{H^{2s}(\Omega)} + \left \|  \tfrac{\partial L}{\partial u} (\cdot,\bar u) \right \|_{H^{ \frac{1}{2} - s - \epsilon}(\Omega)} =: \mathfrak{C}.
\]
A nonlinear interpolation argument yields $\bar z \in H^{4s}(\Omega)$ with a similar estimate.

\noindent  
\boxed{3.3.1} $s \in [\tfrac{1}{10},\tfrac{1}{6})$.  Observe that $4s>\frac{1}{2} -s-\epsilon$ for $\epsilon >0$ being arbitrarily small. Invoke Proposition \ref{pro:state_regularity_smooth} with $t = \frac{1}{2} - s - \epsilon$ to obtain $\bar u \in H^{s+ \frac{1}{2} - \epsilon}(\Omega)$ with 
\[
\| \bar u \|_{H^{s + \frac{1}{2} - \epsilon}(\Omega)} 
\lesssim
\| \bar z\|_{H^{4s}(\Omega)}
+
\| \bar u \|_{H^{4s}(\Omega)}
+
\| a(\cdot,0) \|_{H^{\frac{1}{2} -s-\epsilon}(\Omega)}
\lesssim 
\mathfrak{C}
+
\| a(\cdot,0) \|_{H^{\frac{1}{2} -s-\epsilon}(\Omega)}.
\]
Invoke Proposition \ref{pro:state_regularity_smooth} again to deduce that  $\bar p, \bar z \in H^{s+ \frac{1}{2} - \epsilon}(\Omega)$.

\noindent  
\boxed{3.3.2} $s \in (0,\tfrac{1}{10})$.  Invoke Proposition \ref{pro:state_regularity_smooth} with $t = 4s$ to obtain $\bar u \in H^{6s}(\Omega)$. Since $\bar p,\tfrac{\partial a }{\partial u} (\cdot,\bar u) \in H^{4s}(\Omega)$, an application of Proposition \ref{pro:state_regularity_smooth} yields $\bar p, \bar z \in H^{6s}(\Omega)$.

\noindent  
\boxed{3.3.2.1} $s \in [\tfrac{1}{14},\tfrac{1}{10})$. Observe that $6s>\frac{1}{2} -s-\epsilon$ for $\epsilon >0$ being arbitrarily small. Invoke Proposition \ref{pro:state_regularity_smooth} with $t = \frac{1}{2} - s - \epsilon$ to obtain $\bar u,\bar p, \bar z \in H^{s+ \frac{1}{2} - \epsilon}(\Omega)$.

\noindent  
\boxed{3.3.2.2} $s \in (0,\tfrac{1}{14})$.  Invoke Proposition \ref{pro:state_regularity_smooth} with $t = 6s$ to obtain $\bar u \in H^{8s}(\Omega)$. Since $\bar p,\tfrac{\partial a }{\partial u} (\cdot,\bar u) \in H^{6s}(\Omega)$, Proposition \ref{pro:state_regularity_smooth} also yields $\bar p, \bar z \in H^{8s}(\Omega)$.

From this procedure we note that, at every step, there is a regularity gain.  Consequently, after a finite number of steps, which is proportional to $s^{-1}$, we can conclude that the desired regularity results hold. This concludes the proof.
\end{proof}

\section{Finite element approximation of fractional semilinear PDEs}
\label{sec:fem}
In this section, we analyze the convergence properties of suitable finite element discretizations and derive, when possible, a priori error estimates. For analyzing convergence properties, it will be sufficient to assume that $\Omega$ is an open and bounded Lipschitz polytope. However, additional assumptions on $\Omega$ will be imposed for deriving error estimates: $\Omega$ is smooth and convex; convexity being assumed for simplicity. Since in this case $\Omega$ cannot be meshed exactly, we consider curved simplices to discretize $\Omega \setminus \Omega_h$; $\Omega_h$ being a suitable polytopal domain that \emph{approximates} $\Omega$.

For the sake of brevity, we restrict the presentation to open and bounded domains $\Omega \subset \mathbb{R}^n$ $(n \geq 2)$ such that $\partial \Omega \in C^2$; for Lipschitz polytopes the presentation is simpler (see Remark \ref{rem:polytopes}). We follow \cite[Section 5.2]{MR773854} and consider a family of open, bounded, and convex polytopal domains $\{ \Omega_h \}_{h>0}$, based on a family of quasi-uniform partitions made of closed simplices $\{ \T_h \}_{h>0}$, that approximate $\Omega$ in the following sense:  
\begin{equation}
\mathcal{N}_h \subset \bar \Omega_h,
\quad
\mathcal{N}_h \cap \partial \Omega_h \subset \partial \Omega,
\quad
|\Omega \setminus \Omega_h| \lesssim h^2.
\label{eq:properties_of_Omegah}
\end{equation}
Here, $h = \max_{T \in \T_h} h_T$ denotes the mesh-size of the quasi-uniform partition $\T_h = \{ T \}$, where $h_T = \mathrm{diam}(T)$, and $\mathcal{N}_h$ corresponds to the set of all nodes of the mesh $\T_h$. We shall also assume that $\Omega$ is convex so that $\Omega_h \subset \Omega$ for every $h>0$.

Given a mesh $\T_h$, we define the finite element space of continuous piecewise polynomials of degree one as
\begin{equation}
\V_h = \left\{ v_h \in C^0( \overline {\Omega} ): {v_h}_{|T} \in \mathbb{P}_1(T) \ \forall T \in \T_h, \ v_{h} = 0 \textrm{ on } \overline \Omega \setminus \Omega_h \right\}.
\label{eq:defFESpace}
\end{equation}
Note that discrete functions are trivially extended by zero to $\Omega^c$ and that we enforce a classical homogeneous Dirichlet boundary condition at the degrees of freedom that are located at the boundary of \(\Omega_h\). 

\begin{remark}[polytopes]
\rm
If $\Omega$ is a Lipschitz polytope the previous construction is not necessary:
%
$\Omega = \Omega_h$
and 
$\V_h = \{ v_h \in C^0( \overline{\Omega} ): {v_h}_{|T} \in \mathbb{P}_1(T) \ \forall T \in \T_h  \}$.
\label{rem:polytopes}
\end{remark}

\subsection{The discrete problem}
We introduce the following finite element approximation of problem \eqref{eq:weak_semilinear_pde}: Find $\mathsf{u}_h \in \V_h$ such that
\begin{equation}
 \label{eq:discrete_semilinear_pde}
  \mathcal{A} (\mathsf{u}_h,v_h)  +  \int_{\Omega_h} a(x,\mathsf{u}_h(x)) v_h(x) \mathrm{d}x = \int_{\Omega_h} f(x) v_h(x) \mathrm{d}x \quad \forall v_h \in \V_h.
\end{equation}
Let $r>n/2s$ and $f \in L^r(\Omega)$. Let $a = a(x,u) : \Omega \times \mathbb{R} \rightarrow \mathbb{R}$ be a Carath\'eodory function that is monotone increasing in $u$. Assume, in addition, that $a$ satisfies \eqref{eq:assumption_on_phi_state_equation} and $a(\cdot,0) \in L^r(\Omega)$. Withing this setting, Theorem \ref{thm:stata_equation_well_posedness} guarantees that the continuous problem \eqref{eq:weak_semilinear_pde} admits a unique solution $u \in \tilde H^s(\Omega) \cap L^{\infty}(\Omega)$ satisfying \eqref{eq:stability}. Since $\mathcal{A}$ is coercive and $a$ is monotone increasing in $u$, an appllication of Brouwer's fixed point theorem \cite[Proposition 2.6]{MR816732} yields the existence of a unique solution for \eqref{eq:discrete_semilinear_pde}; see also the proof of \cite[Theorem 26.A]{MR1033498}. In addition, $\| \mathsf{u}_h \|_s \lesssim \| f \|_{H^{-s}(\Omega)}$ for every $h>0$.

\subsection{Regularity estimates}
Before deriving error estimates, it is of fundamental importance the understanding of regularity estimates for the solution of \eqref{eq:weak_semilinear_pde}. 

\begin{theorem}[regularity estimates: $s \in (0,1)$]
\label{thm:regularity_space_state_equation}
Let $n \geq 2$, $s \in (0,1)$, and $\Omega$ be a domain such that $\partial \Omega \in C^{\infty}$. Assume, in addition, that $a$ is locally Lipschitz with respect to the second variable. If both $a(\cdot,0)$ and $f$ belong to $H^{1/2 - s - \epsilon}(\Omega)$, with $\epsilon$ arbitrarily small, then
$u \in H^{s + 1/2 - \epsilon}(\Omega)$. 
\end{theorem}
\begin{proof}
The proof follows along the same lines of Theorem \ref{thm:regularity_space}. For brevity, we skip the details.
\end{proof}

\subsection{Error estimates}
We now present error estimates. In doing so, we will assume, in addition, that there exists $\phi \in L^{\mathfrak{r}}(\Omega)$, with $\mathfrak{r} = n/2s$, such that
\begin{equation}
|a(x,u) - a(x,v)| \leq |\phi(x)| |u-v| ~\textrm{a.e.}~x \in \Omega,~u,v \in \mathbb{R}.
\label{eq:assumption_on_a_phi}
\end{equation}

\begin{theorem}[error estimates]
Let $n \geq 2$, $s \in (0,1)$, and $r>n/2s$. Let $\Omega$ be an open and bounded domain with Lipschitz boundary. Assume that $a$ is as in the statement of Theorem \ref{thm:stata_equation_well_posedness}. Assume, in addition, that $a$ satisfies \eqref{eq:assumption_on_a_phi}. Let $u \in \tilde H^s(\Omega)$ be the solution to \eqref{eq:weak_semilinear_pde} and let $\mathsf{u}_h \in \mathbb{V}_h$ be its finite element approximation obtained as the solution to \eqref{eq:discrete_semilinear_pde}. Then, we have the quasi--best approximation result
\begin{equation}
\label{eq:error_estimate_semilinear_s}
\| u - \mathsf{u}_h \|_{s} \lesssim \| u - v_h\|_{s} \quad \forall v_h \in \mathbb{V}_h.
\end{equation}
If, in addition,  $\Omega$ is smooth and convex, $a$ is locally Lipschitz with respect to the second variable, and $a(\cdot,0), f  \in H^{1/2 - s - \epsilon}(\Omega)$, with $\epsilon>0$ arbitrarily small, then 
\begin{equation}
\label{eq:error_estimate_semilinear_s_final}
\| u - \mathsf{u}_h \|_{s} \lesssim h^{\frac{1}{2} - \epsilon } \| u \|_{H^{s + 1/2 - \epsilon}(\Omega)}.
\end{equation} 
If, in addition, \eqref{eq:assumption_on_a_phi} holds with $\mathfrak{r} = n/s$, then
\begin{equation}
\label{eq:error_estimate_semilinear_s_final_L2}
\| u - \mathsf{u}_h \|_{L^2(\Omega)} \lesssim h^{\vartheta + \frac{1}{2} - \epsilon} \| u \|_{H^{s + 1/2 - \epsilon}(\Omega)}.
\end{equation} 
Here, $\vartheta = \min \{ s , \tfrac{1}{2} - \epsilon \}$ with $\epsilon >0$ being arbitrarily small. In all three estimates the hidden constant is independent of $u$, $\mathsf{u}_h$, and $h$.
\label{thm:error_estimates_state_equation}
\end{theorem}
\begin{proof}
Since $a$ is monotone increasing in the second variable, we obtain
\begin{align*}
\| u - \mathsf{u}_h \|_s^2 & = \mathcal{A}(u - \mathsf{u}_h,u - \mathsf{u}_h) \leq \mathcal{A}(u - \mathsf{u}_h,u - \mathsf{u}_h)  + (a(\cdot,u) - a(\cdot,\mathsf{u}_h), u - \mathsf{u}_h)_{L^2(\Omega)}
\\
& =  \mathcal{A}(u - \mathsf{u}_h,u - v_h)  + (a(\cdot,u) - a(\cdot,\mathsf{u}_h), u - v_h)_{L^2(\Omega)},
\quad
v_h \in \mathbb{V}_h
\end{align*}
upon utilizing Galerkin orthogonality. Invoke estimate \eqref{eq:assumption_on_a_phi} and the Sobolev embedding $H^s(\Omega) \hookrightarrow L^{\mathfrak{q}}(\Omega)$ with $\mathfrak{q} \leq 2n/(n-2s)$ to obtain \eqref{eq:error_estimate_semilinear_s}. 

Assume now that $\Omega$ is smooth and convex so Theorem \ref{thm:regularity_space_state_equation} applies; convexity being assumed for simplicity. To bound $\| u - v_h \|_s$ we first invoke \cite[Theorem 3.33]{McLean}:
\[
\| u -v_h \|_s \lesssim \| u - v_h \|_{H^s(\Omega)}
\quad
\forall v_h \in \mathbb{V}_h,
\qquad
 s \in (0,1) \setminus \{ \tfrac{1}{2} \}.
\]
The second ingredient is the localization of fractional order Sobolev seminorms  \cite{MR1752263,MR1930387}:
\[
|v|^2_{H^s(\Omega)} \leq \sum_{T} \left[ \int_{T} \int_{S_T} \frac{|v(x) - v(y)|^2}{|x-y|^{n+2s}} \mathrm{d}y \mathrm{d}x + \frac{\mathfrak{c}}{ s h_T^{2s} } \| v \|^2_{L^2(T)} \right], 
\quad
s \in (0,1),
\quad
\mathfrak{c}>0,
\]
for $v \in H^s(\Omega)$; $S_T $ denotes a suitable patch associated to $T$. We stress that curved domains/simplices are also handled in \cite{MR1752263,MR1930387}. It thus suffices to note that, if $\mathfrak{T}$ denotes a boundary curved simplex, the fact that $u \in \tilde H^s(\Omega) \cap H^{s+1/2 - \epsilon}(\Omega)$, with $\epsilon >0$ arbitrarily small, implies
\[
\| u - \mathsf{u}_h \|_{L^2(\mathfrak{T})} \leq  \| u \|_{L^2(\Omega \setminus \Omega_h)} \lesssim h^{2\upsilon}  \| u \|_{H^{\upsilon}(\Omega)}, \quad \upsilon = \min \{1,s+1/2 - \epsilon \},
\]
which follows from interpolating \cite[estimate (5.2.18)]{MR773854} and $\| v \|_{L^2(\Omega \setminus \Omega_h)} \leq \| v \|_{L^2(\Omega)}$. On the other hand, if $v \in \tilde H^s(\Omega) \cap H^{s+1/2 - \epsilon}(\Omega)$, with $\epsilon>0$ arbitrarily small, then
\[
\int_{\mathfrak{T}} \int_{S_{\mathfrak{T}}} \frac{|v(x) - v(y)|^2}{|x-y|^{n+2s}} \mathrm{d}y \mathrm{d}x
\leq |v|^2_{H^s(S_\mathfrak{T})} \lesssim h^{2(1/2 - \epsilon)} \| v \|^2_{H^{s+1/2 - \epsilon}(\Omega)}.
\]
We thus utilize interpolation error estimates for the Scott--Zhang operator \cite[Proposition 3.6]{MR3893441} and Theorem \ref{thm:regularity_space_state_equation} to arrive at the estimate \eqref{eq:error_estimate_semilinear_s_final}; see \cite[Section 3.2]{MR3893441} for details and the particular treatment of the case $s = 1/2$. 

The error estimate in $L^2(\Omega)$ follows from duality. Define 
$ 0 \leq  \chi \in L^{\mathfrak{r}}(\Omega)$ by
\[
\chi(x) = \frac{a(x,u(x)) - a(x,\mathsf{u}_h(x))}{u(x) - \mathsf{u}_h(x)}
~\mathrm{if}~u(x) \neq \mathsf{u}_h(x),
\quad
\chi(x) = 0 
~\mathrm{if}~u(x) = \mathsf{u}_h(x).
\]
Let $\mathfrak{z} \in \tilde H^s(\Omega)$ be the solution to
$
\mathcal{A}(v,\mathfrak{z}) + (\chi \mathfrak{z} ,v)_{L^2(\Omega)} = \langle \mathfrak{f} , v \rangle
$
for all $v \in \tilde H^s(\Omega)$; $\mathfrak{f} \in H^{-s}(\Omega)$. Let $\mathfrak{z}_h$ be the finite element approximation of $\mathfrak{z}$ within $\mathbb{V}_h$. Thus, 
\begin{align*}
\langle \mathfrak{f} , u - \mathsf{u}_h \rangle & = \mathcal{A}(u-\mathsf{u}_h,\mathfrak{z}) + (\chi \mathfrak{z}, u - \mathsf{u}_h)_{L^2(\Omega)} = \mathcal{A}(u-\mathsf{u}_h,\mathfrak{z} - \mathfrak{z}_h)
+ \mathcal{A}(u-\mathsf{u}_h,\mathfrak{z}_h)
\\
& + (\chi \mathfrak{z}, u - \mathsf{u}_h)_{L^2(\Omega)} 
= \mathcal{A}(u-\mathsf{u}_h,\mathfrak{z} - \mathfrak{z}_h)
+ (a(\cdot,u) - a(\cdot,\mathsf{u}_h), \mathfrak{z} - \mathfrak{z}_h)_{L^2(\Omega)} 
\\
& \leq
\| u - \mathsf{u}_h \|_s \| \mathfrak{z} - \mathfrak{z}_h \|_s + \| \phi \|_{L^{r}(\Omega)} \| u - \mathsf{u}_h \|_{L^{q}(\Omega)} \| \mathfrak{z} - \mathfrak{z}_h \|_{L^{q}(\Omega)}.
\end{align*}
Here, $q$ satisfies $2q^{-1} + r^{-1} = 1$, i.e., $q = 2n/(n-2s)$. 
Set $\mathfrak{f} = u - \mathsf{u}_h \in L^2(\Omega)$. Notice that, since $\phi \in L^{\mathfrak{r}}(\Omega)$, with $\mathfrak{r} = n/s$, $\chi  \mathfrak{z}$ belong to $L^2(\Omega)$. We can thus invoke Proposition \ref{pro:state_regularity_smooth} with $t=0$ to obtain $\| \mathfrak{z} \|_{H^{s+\theta}(\Omega)} \lesssim \|  u - \mathsf{u}_h \|_{L^2(\Omega)}$. Consequently,
\[
\| u - \mathsf{u}_h \|^2_{L^2(\Omega)} \lesssim  \| u - \mathsf{u}_h \|_s \| \mathfrak{z} - \mathfrak{z}_h \|_s \lesssim h^{ \frac{1}{2} - \epsilon} \|u\|_{H^{s+\frac{1}{2} - \epsilon}(\Omega)} h^{\vartheta} \| u - \mathsf{u}_h  \|_{L^{2}(\Omega)},
\]
where  $\vartheta = \min \{s,1/2-\epsilon\}$ and $\epsilon >0$ is arbitrarily small.  This concludes the proof.
\end{proof}

\subsection{Convergence properties}
Let  $1<p<\infty$ and let $\{ f_h \}_{h>0}$ be a sequence such that $f_h \in L^p(\Omega_h)$. We will say that $f_h \rightharpoonup f$ in $L^{p}(\Omega)$ as $h \downarrow0$ if $f \in L^p(\Omega)$ and 
\begin{equation}
\int_{\Omega_h} f_h(x) v(x) \mathrm{d}x \rightarrow \int_{\Omega} f(x) v(x) \mathrm{d}x
\quad
\forall v \in L^{q}(\Omega),
\quad
h \downarrow 0,
\quad
p^{-1} + q^{-1} = 1.
\label{eq:weak_convergence_h}
\end{equation}
If $p = \infty$, we will say that  $f_h \mathrel{\ensurestackMath{\stackon[1pt]{\rightharpoonup}{\scriptstyle\ast}}} f$ in $L^{\infty}(\Omega)$ if $f \in L^{\infty}(\Omega)$ and \eqref{eq:weak_convergence_h} holds for every $v \in L^1(\Omega)$. Observe that, upon considering a suitable extension of $f_h$ to $\Omega \setminus \Omega_h$, $f_h$ can be understood as an element of $L^{p}(\Omega)$. Since $| \Omega \setminus \Omega_h | \rightarrow 0$ as $h \downarrow 0$, \eqref{eq:weak_convergence_h} is equivalent to $\int_{\Omega} \tilde{f}_h(x) v(x) \mathrm{d}x \rightarrow \int_{\Omega} f(x) v(x) \mathrm{d}x$, for instance, for $\{ \tilde{f_h} \}_{h>0}$ being a uniformly bounded extension of $\{ f_h \}_{h>0}$ to $\Omega$ or an extension independent of $h$.

\begin{remark}[polytopes]
\rm
If $\Omega$ is a Lipschitz polytope, then \eqref{eq:weak_convergence_h} reduces to the standard concept of weak convergence in $L^p(\Omega)$ because $\Omega_ h = \Omega$ for every $h>0$.
\label{rem:polytopes_weak_convergence}
\end{remark}

\begin{proposition}[convergence]
Let $n \geq 2$, $s \in (0,1)$, and $r>n/2s$. Let $\Omega$ be an open, bounded, and convex domain such that $\partial \Omega \in C^{2}$. Assume that $a$ is as in the statement of Theorem \ref{thm:stata_equation_well_posedness} and satisfies, in addition, \eqref{eq:assumption_on_a_phi}. Let $u \in \tilde H^s(\Omega)$ solves \eqref{eq:weak_semilinear_pde}.  Let $\mathfrak{u}_h  \in \mathbb{V}_h$ be the solution to \eqref{eq:discrete_semilinear_pde} with $f$ replaced by $f_h \in L^{r}(\Omega_h)$. Then,
\[
f_h \rightharpoonup f \textrm{ in }L^{r}(\Omega)
\implies
\mathfrak{u}_h \rightarrow u \textrm{ in } L^{\mathfrak{t}}(\Omega),
\quad
h \downarrow 0,
\quad
\mathfrak{t} \leq 2n/(n-2s).
\]
Here, $f_h \rightharpoonup f$ in $L^r(\Omega)$ is understood in the sense of \eqref{eq:weak_convergence_h}.
\label{eq:proposition_convergence}
\end{proposition}
\begin{proof}
We begin with a simple application of the triangle inequality and write
\[
\|  u -  \mathfrak{u}_h  \|_{L^{\mathfrak{t}}(\Omega)} 
\leq 
\|   u - \mathsf{u}_h  \|_{L^{\mathfrak{t}}(\Omega)}
+
\|  \mathsf{u}_h  -  \mathfrak{u}_h  \|_{L^{\mathfrak{t}}(\Omega)},
\quad
\mathfrak{t} \leq  2n/(n-2s),
\]
where $\mathsf{u}_h$ denotes the solution to \eqref{eq:discrete_semilinear_pde}. Since $H^s(\Omega) \hookrightarrow L^{\mathfrak{q}}(\Omega)$ for $\mathfrak{q} \leq 2n/(n-2s)$, the quasi--best approximation estimate \eqref{eq:error_estimate_semilinear_s} yields $\|   u - \mathsf{u}_h \|_{L^{\mathsf{q}}(\Omega)} \lesssim \| u - v_h \|_s$ for an arbitrary $v_h \in \mathbb{V}_h$. A density argument as in \cite[Theorem 3.2.3]{MR0520174} reveals the convergence result $\| u - \mathsf{u}_h \|_{L^{\mathsf{q}}(\Omega)} \rightarrow 0$ as $h \downarrow 0$. 

To control $\|  \mathsf{u}_h  -  \mathfrak{u}_h  \|_{L^{\mathfrak{t}}(\Omega)}$ we invoke the problems that $\mathsf{u}_h$ and $ \mathfrak{u}_h$ solve:
\begin{multline}
\|  \mathsf{u}_h  -  \mathfrak{u}_h  \|^2_s = \mathcal{A}(\mathsf{u}_h  -  \mathfrak{u}_h , \mathsf{u}_h  -  \mathfrak{u}_h   ) = (f - f_h, \mathsf{u}_h  -  \mathfrak{u}_h)_{L^2(\Omega)} 
\\
-  \left( a(\cdot,\mathsf{u}_h) - a(\cdot,\mathfrak{u}_h), \mathsf{u}_h   -  \mathfrak{u}_h \right)_{L^2(\Omega)}
\leq
\| f - f_h \|_{H^{-s}(\Omega)} \| \mathsf{u}_h  -  \mathfrak{u}_h \|_s.
\end{multline}
This immediately yields $\|  \mathsf{u}_h  -  \mathfrak{u}_h  \|_{L^{\mathfrak{t}}(\Omega)} \lesssim \| f - f_h \|_{H^{-s}(\Omega)}$. Since $f_h \rightharpoonup f$ in $L^{r}(\Omega)$ we can thus obtain that $ \|  \mathsf{u}_h  -  \mathfrak{u}_h  \|_{L^{\mathfrak{t}}(\Omega)} \rightarrow 0$ as $h \downarrow 0$. This concludes the proof.
\end{proof}

\begin{remark}[convergence on polytopes]
\rm
The result of Proposition \ref{eq:proposition_convergence} can also be obtained for Lipschitz polytopes; observe that the involved arguments do not utilize further regularity beyond what is natural for the problem: $u \in \tilde H^s(\Omega) \cap L^{\infty}(\Omega)$.
\label{rem:polytopes_state_convergence}
\end{remark}

\section{Finite element approximation of the adjoint equation}
\label{sec:fem_adjoint}
We begin the section by introducing the following approximation of \eqref{eq:adj_eq}: Find $q_h \in \V_h$ such that
\begin{equation}
 \label{eq:discrete_adjoint}
  \mathcal{A} (v_h,q_h)  + \left( \tfrac{\partial a}{\partial u}(\cdot,u) q_h, v_h \right)_{L^2(\Omega)} 
  = 
  \left(\tfrac{\partial L}{\partial u} (\cdot,u), v_h \right)_{L^2(\Omega)} \quad \forall v_h \in \V_h.
\end{equation}
Here, $u \in \tilde H^s(\Omega) \cap L^{\infty}(\Omega)$ denotes the unique solution to \eqref{eq:weak_st_eq}. Observe that assumption \textnormal{\ref{B2}} yields $\partial L/ \partial u(\cdot,u) \in L^r(\Omega)$ for $r>n/2s$ while assumption \textnormal{\ref{A2}} guarantees that $\partial a/\partial u (x,u) \geq 0$ for a.e.~$x \in \Omega$ and for all $u \in \mathbb{R}$. The existence of a unique discrete solution $q_h \in \mathbb{V}_h$ to problem \eqref{eq:discrete_adjoint} is thus immediate.

We present the following error estimates.
\begin{theorem}[error estimate]
Let $n \geq 2$ and $s \in (0,1)$. Let $\Omega$ be a convex domain such that $\partial \Omega \in C^{\infty}(\Omega)$. Assume that \ref{A1}--\ref{A3}, \ref{B1}--\ref{B2}, and \ref{C1}--\ref{C2} hold.  Let $p \in \tilde H^s(\Omega)$ be the solution to \eqref{eq:adj_eq} and let $q_h \in \mathbb{V}_h$ be its finite element approximation obtained as the solution to \eqref{eq:discrete_adjoint}. Then, we have the error estimates
\begin{equation}
\| p - q_h \|_{s} \lesssim h^{\frac{1}{2} - \epsilon} \| p \|_{H^{s + \frac{1}{2} - \epsilon}(\Omega)},
\qquad
\| p - q_h \|_{L^2(\Omega)} \lesssim h^{ \vartheta + \frac{1}{2} - \epsilon} \| p \|_{H^{s + \frac{1}{2} -\epsilon}(\Omega)},
\label{eq:estimate_adjoint_equation_s_1}
\end{equation}
where $\vartheta = \min \{s,1/2 - \epsilon\}$ and $\epsilon >0$ is arbitrarily small. In both estimates, the hidden constant is independent of $p$, $q_h$, and $h$.
\label{thm:error_estimates_adjoint_equation}
\end{theorem}
\begin{proof}
Notice that, within our setting, Galerkin orthogonality reads as follows: for every $v_h \in \mathbb{V}_h$, we have 
$
\mathcal{A}(v_h,p - q_h) + \left( \tfrac{\partial a}{\partial u}(\cdot,u) (p - q_h), v_h \right)_{L^2(\Omega)} = 0.
$
Thus,
\begin{equation*}
\begin{aligned}
\| p - q_h \|_{s}^2 
& =
 \mathcal{A}(p - q_h,p ) + \left(  \tfrac{\partial a}{\partial u}(\cdot,u) (p -q_h ), q_h \right)_{L^2(\Omega)}
 \\
 & =
  \mathcal{A}(p - q_h,p -v_h) +   \left(  \tfrac{\partial a}{\partial u}(\cdot,u) (p-q_h ) ,q_h - v_h \right)_{L^2(\Omega)}.
\end{aligned}
\end{equation*}
Since $\partial a/\partial u (x,u) \geq 0$ for a.e.~$x \in \Omega$ and $u \in \mathbb{R}$, we invoke \ref{A3} and the fact that $u \in \tilde H^s(\Omega) \cap L^{\infty}(\Omega)$ to obtain 
$
\| p - q_h \|_{s}^2 \leq\| p - q_h \|_{s} \| p - v_h \|_{s} + C_{\mathfrak{m}}\| p - q_h \|_{L^2(\Omega)} \| p - v_h \|_{L^2(\Omega)}.
$
This estimate yields the quasi--best approximation property: $\| p - q_h \|_{s} \lesssim \| p - v_h \|_{s}$ for every $v_h \in \mathbb{V}_h$. The left--hand side estimate in  \eqref{eq:estimate_adjoint_equation_s_1} thus follows from the arguments developed in the proof of Theorem \ref{thm:error_estimates_state_equation}. We note that, in view of Proposition \ref{pro:state_regularity_smooth} and assumptions \ref{C1} and \ref{C2}, a bootstrap argument, as the one developed in the proof of Theorem \ref{thm:regularity_space}, reveals that $p \in H^{s+\frac{1}{2}-\epsilon}(\Omega)$ for every $\epsilon >0$ arbitrarily small. The right--hand side estimate in \eqref{eq:estimate_adjoint_equation_s_1} follows from a duality argument.
\end{proof}

In what follows, we will operate under the \emph{assumption} that discrete solutions $u_h$ to problem \eqref{eq:discrete_semilinear_pde} are uniformly bounded in $L^{\infty}(\Omega)$, i.e.,
 \begin{equation}
 \exists C >0: \quad \| u_h \|_{L^{\infty}(\Omega)} \leq C \quad \forall h > 0.
 \label{eq:u_h_bounded}
 \end{equation}

Let $u_h$ be the solution to \eqref{eq:discrete_semilinear_pde} with $f$ replaced by $z_h$; $z_h$ being an arbitrary piecewise constant function over $\T_h$. Let $p_h \in \mathbb{V}_h$ be the unique solution to 
\begin{equation}
 \label{eq:discrete_adjoint_2}
 \quad
  \mathcal{A} (v_h, p_h)  + \left( \tfrac{\partial a}{\partial u}(\cdot,u_h) p_h, v_h \right)_{L^2(\Omega)} 
  = 
  \left(\tfrac{\partial L}{\partial u} (\cdot,u_h), v_h \right)_{L^2(\Omega)} \quad \forall v_h \in \V_h.
\end{equation}

We now derive estimates for the error $p - p_h$. To accomplish this task, we first define $q$ as the solution to the following problem: Find $q \in \tilde H^s(\Omega)$ such that
\begin{equation}
\label{eq:q}
  \mathcal{A} (v,q)  + \left( \tfrac{\partial a}{\partial u}(\cdot,u_h) q, v \right)_{L^2(\Omega)} 
  = 
 \left(\tfrac{\partial L}{\partial u} (\cdot,u_h), v \right)_{L^2(\Omega)} \quad \forall v \in \tilde H^s(\Omega).
\end{equation}
In view of the assumptions on the data and \eqref{eq:u_h_bounded}, problems \eqref{eq:discrete_adjoint_2} and \eqref{eq:q} are well-defined. In particular, we have $q \in \tilde H^s(\Omega) \cap L^{\infty}(\Omega)$. Observe that $p_h$ can be seen as the finite element approximation of $q$ within $\mathbb{V}_h$. Consequently, Theorem \ref{thm:error_estimates_adjoint_equation} yields
\begin{equation}
\| q - p_h \|_s \lesssim h^{\frac{1}{2}-\epsilon} \| q \|_{H^{s+ \frac{1}{2}-\epsilon}(\Omega)},
\quad
\| q - p_h \|_{L^2(\Omega)} \lesssim h^{\vartheta + \frac{1}{2}-\epsilon} \| q \|_{H^{s + \frac{1}{2} - \epsilon}(\Omega)},
\label{eq:q-ph}
\end{equation}
where $\vartheta = \min \{ s , 1/2 - \epsilon \}$ and $\epsilon$ is arbitrarily small. Observe that \ref{C1} and \ref{C2} guarantee that $q \in H^{s + \frac{1}{2}-\epsilon}(\Omega)$. We also define the variable $y$ to be such that
\begin{equation}
\label{eq:mathsf_u}
y \in \tilde H^s(\Omega):
\quad
  \mathcal{A}(y,v)  +  \langle a(\cdot,y),v \rangle = \langle z_h , v \rangle 
  \quad \forall v \in \tilde H^{s}(\Omega).
\end{equation}
Since $z_h \in L^{\infty}(\Omega)$, Theorem \ref{thm:stata_equation_well_posedness} yields the well-posedness of \eqref{eq:mathsf_u} and $y \in \tilde H^s(\Omega) \cap L^{\infty}(\Omega)$. On the other hand, since $z_h \in H^{\frac{1}{2}-\epsilon}(\Omega)$, for every $\epsilon >0$, a bootstrapping argument and \ref{C1} allow us to conclude that $y \in H^{s+ \frac{1}{2}-\epsilon}(\Omega)$ for every $\epsilon >0$.

We present the following error estimates.
\begin{theorem}[error estimates]
Let the assumptions of Theorem \ref{thm:error_estimates_adjoint_equation} hold. Assume, in addition, that $\partial L/\partial u$ is locally Lipschitz with respect to the second variable and that $a$ satisfies \eqref{eq:assumption_on_a_phi}. Let $p \in \tilde H^s(\Omega)$ be the solution to \eqref{eq:adj_eq} and let $p_h \in \mathbb{V}_h$ be the solution to \eqref{eq:discrete_adjoint_2}. Then, we have the error estimate
\begin{equation}
\| p - p_h \|_{s} \lesssim h^{\frac{1}{2} - \epsilon} + \| z - z_h \|_{L^2(\Omega)}.
\label{eq:estimate_adjoint_equation_s_2}
\end{equation}
If, in addition, $a$ satisfies \eqref{eq:assumption_on_a_phi} with $\mathfrak{r} = n/s$, we also have the error estimate
\begin{equation}
\| p - p_h \|_{L^2(\Omega)} \lesssim h^{ \vartheta + \frac{1}{2} - \epsilon} + \| z - z_h \|_{L^2(\Omega)},
\label{eq:estimate_adjoint_equation_s_3}
\end{equation}
where $\vartheta = \min \{s,1/2 - \epsilon\}$. In both estimates, $\epsilon >0$ is arbitrarily small and the hidden constant is independent of $h$.
\label{thm:error_estimates_adjoint_equation_2}
\end{theorem}
\begin{proof}
We begin with a simple application of the triangle inequality: $\| p - p_h \|_{s} \leq \| p - q \|_{s} + \| q - p_h \|_{s}$. The control of $\| q - p_h \|_{s}$ follows from \eqref{eq:q-ph}. To bound $\| p - q \|_{s}$, we first observe that, for every $v \in \tilde H^s(\Omega)$, we have
\begin{multline*}
p- q \in \tilde H^s(\Omega):
\quad
  \mathcal{A} (v,p-q)  + \left( \tfrac{\partial a}{\partial u}(\cdot,u) (p-q), v \right)_{L^2(\Omega)} 
 \\
  = 
   \left( \left[ \tfrac{\partial a}{\partial u}(\cdot,u_h) -  \tfrac{\partial a}{\partial u}(\cdot,u) \right]q, v \right)_{L^2(\Omega)} 
   +
  \left( \tfrac{\partial L}{\partial u} (\cdot,u) - \tfrac{\partial L}{\partial u} (\cdot,u_h), v \right)_{L^2(\Omega)}.
\end{multline*}
Since $\tfrac{\partial a}{\partial u}$ and $ \tfrac{\partial L}{\partial u} $ are locally Lipschitz with respect to the second variable we obtain
$
\| p - q \|_s \lesssim \| u - u_h \|_{L^2(\Omega)} ( 1 + \| q \|_{L^{\infty}(\Omega)}).
$
It thus suffices to bound $\| u - u_h \|_{L^2(\Omega)}$. To do this, we write $ \| u - u_h \|_{L^2(\Omega)} \leq \| u - y \|_{L^2(\Omega)} + \| y - u_h \|_{L^2(\Omega)}$, where $y$ denotes the solution to \eqref{eq:mathsf_u}. An application of Theorem \ref{thm:error_estimates_state_equation} yields the control of $\| y - u_h \|_{L^2(\Omega)}$. The control of $\| u - y \|_{L^2(\Omega)}$ follows from writing the problem that 
$u - y $ solves and utilizing assumptions \ref{A1}--\ref{A3}:
$ \| u -y \|_{L^2(\Omega)} \lesssim \| z - z_h \|_{L^2(\Omega)}$. A collection of the derived estimates yield \eqref{eq:estimate_adjoint_equation_s_2}. The proof of \eqref{eq:estimate_adjoint_equation_s_3} follows similar arguments.
\end{proof}

\section{Finite element approximation for the optimal control problem}
\label{sec:fem_control}

In this section, we propose a finite element discretization scheme for our control problem. We analyze convergence properties and derive, when possible, error estimates. To accomplish this task, we operate within the discrete setting introduced in section \ref{sec:fem} and introduce, in addition, the finite element space of piecewise constant functions 
\begin{equation}
 \mathbb{Z}_h = \left\{ v_h \in L^{\infty}( \Omega_h ): {v_h}_{|T} \in \mathbb{P}_0(T) \ \forall T \in \T_h  \right\}
 \label{eq:piecewise_constant_functions}
\end{equation}
and the space of discrete admissible controls
$
 \mathbb{Z}_{\mathrm{ad},h} = \mathbb{Z}_{\mathrm{ad}} \cap  \mathbb{Z}_h.
$

\subsection{The discrete optimal control problem}
\label{sec:discrete_optimal_control_problem}

We consider the following discrete counterpart of the continuous optimal control problem \eqref{eq:min}--\eqref{eq:weak_st_eq}: Find
\begin{equation}\label{eq:min_discrete}
\min \{ J_h(u_h,z_h): (u_h,z_h) \in \mathbb{V}_h \times \mathbb{Z}_{ad,h} \}
\end{equation}
subject to the \emph{discrete state equation}
\begin{equation}\label{eq:weak_st_eq_discrete}
\mathcal{A}( u_h, v_h)+\int_{\Omega_h} a(x,u_h(x)) v_h(x) \mathrm{d}x = \int_{\Omega_h} z_h(x) v_h(x) \mathrm{d}x \quad \forall v \in \mathbb{V}_h.
\end{equation}
Here, $J_h: \mathbb{V}_h \times \mathbb{Z}_{ad,h} \ni (u_h,z_h) \mapsto J_h(u_h,z_h):= \int_{\Omega_h} L(x,u_h(x))\mathrm{d}x + \frac{\alpha}{2} \| z_h \|^2_{L^2(\Omega_h)} \in \mathbb{R}$.

We present the following result.

\begin{theorem}[optimal pair and optimality system]
\label{thm:existence_discrete_control}
Let $n \geq 2$ and $s \in (0,1)$. Assume that \textnormal{\ref{A1}}--\textnormal{\ref{A3}} and \textnormal{\ref{B1}}--\textnormal{\ref{B2}} hold. Thus, the discrete optimal control problem \eqref{eq:min_discrete}--\eqref{eq:weak_st_eq_discrete} admits at least one solution $\bar{z}_h \in  \mathbb{Z}_{ad,h}$. In addition, if $\bar{z}_h$ denotes a local minimum for \eqref{eq:min_discrete}--\eqref{eq:weak_st_eq_discrete}, then the triple $(\bar u_h, \bar p_h, \bar z_h) \in \mathbb{V}_h \times \mathbb{V}_h \times \mathbb{Z}_{ad,h}$, with $\bar u_h$ and $\bar p_h$ being the associated optimal state and adjoint state, respectively, satisfies 
\begin{align}
  \mathcal{A} ( \bar u_h,v_h)  +  ( a(\cdot, \bar u_h), v_h )_{L^2(\Omega_h)} & = (\bar z_h, v_h)_{L^2(\Omega_h)}  
\quad
\forall v_h \in \mathbb{V}_h,
 \label{eq:optimal_state_discrete}
\\
\mathcal{A}(v_h,\bar p_h) +  \left(  \tfrac{\partial a}{\partial u}(\cdot, \bar u_h) \bar p_h, v_h \right)_{L^2(\Omega_h)}  & = \left( \tfrac{\partial L}{\partial u}(\cdot,\bar u_h), v_h \right)_{L^2(\Omega_h)}
\quad
\forall v_h \in \mathbb{V}_h,
 \label{eq:optimal_adjoint_state_discrete}
\end{align}
and the variational inequality
\begin{equation}
(\bar p_h + \alpha \bar z_h, z_h - \bar z_h)_{L^2(\Omega_h)} \geq 0 
\end{equation}
for every $z_h \in \mathbb{Z}_{ad,h}$.
\end{theorem}
\begin{proof}
The proof follows from the finite dimensional analog of the arguments elaborated in the proof of Theorems \ref{thm:existence_control} and \ref{thm:optimality_cond}. For brevity, we skip details.
\end{proof}
\subsection{Convergence of discretizations}
\label{sec:convergence}
We begin with the following convergence result: \emph{a sequence $\{ \bar z_h \}_{h>0}$ of global solutions of the discrete optimal control problems \eqref{eq:min_discrete}--\eqref{eq:weak_st_eq_discrete} admits subsequences that converge, as $h \downarrow 0$, to global solutions of the continuous optimal control problem \eqref{eq:min}--\eqref{eq:weak_st_eq}.}
 
\begin{theorem}[convergence]
Let $n \geq 2$ and $s \in (0,1)$. Let $\Omega$ be a Lipschitz polytope satisfying the exterior ball condition. Assume that \ref{A1}--\ref{A3} and \ref{B1}--\ref{B2} hold. Assume that $a = a(x,u)$ satisfies, in addition, \eqref{eq:assumption_on_a_phi} and that $L$ satisfies, in addition, for all $\mathfrak{m}>0$, the estimate $|\frac{\partial L}{\partial u}(x,u)| \leq C_{\mathfrak{m}}$ for a.e.~$x \in \Omega$ and $u \in [-\mathfrak{m},\mathfrak{m}]$.
Let $\bar z_h$, for every $h>0$, be a global solution of the discrete optimal control problem. Then, there exist nonrelabeled subsequences $\{ \bar z_h \}_{h>0}$ such that $\bar z_h \mathrel{\ensurestackMath{\stackon[1pt]{\rightharpoonup}{\scriptstyle\ast}}} \bar{z}$ as $h \downarrow 0$, in $L^{\infty}(\Omega)$, with $\bar z$ being a global solution of \eqref{eq:min}--\eqref{eq:weak_st_eq}. In addition, we have
\begin{equation}
\label{eq:convergence}
\| \bar z - \bar z_h \|_{L^{2}(\Omega)} \rightarrow 0,
\qquad
j_{h}( \bar z_h) \rightarrow j(\bar z), 
\end{equation}
as $h \downarrow 0$.
\label{thm:convergence}
\end{theorem}
\begin{proof}
Since $\{ \bar z_h \}_{h>0}$ is uniformly bounded in $L^{\infty}(\Omega)$, we deduce the existence of a nonrelabeled subsequence $\{ \bar z_h \}_{h>0}$ such that $\bar z_h \mathrel{\ensurestackMath{\stackon[1pt]{\rightharpoonup}{\scriptstyle\ast}}} \bar{z}$ in $L^{\infty}(\Omega)$ as $h \downarrow 0$. In what follow, we prove that $\bar z$ is a global solution of the continuous optimal control problem and that $j_h(\bar z_h) \rightarrow j(\bar z)$ as $h \downarrow 0$.

Let $\tilde z \in \mathbb{Z}_{ad}$ be a global solution of \eqref{eq:min}--\eqref{eq:weak_st_eq}. Define $\tilde p$ as the solution to \eqref{eq:adj_eq}, with $u$ replaced by $\tilde u := \mathcal{S} \tilde z$, and $\tilde z_h \in \mathbb{Z}_{ad,h}$ by $\tilde z_h|_{T} := \int_{T} \tilde z(x) \mathrm{d}x / |T|$ for $T \in \T_h$. Observe that, since \ref{A3} holds and $\tilde p, \partial L/\partial u(\cdot,\tilde{u}) \in L^{\infty}(\Omega)$, we deduce that $\partial L/\partial u(\cdot, \tilde u) - \partial a/\partial u(\cdot,\tilde u) \tilde p \in L^{\infty}(\Omega)$.  In view of the fact that $\Omega$ is Lipschitz and satisfies the exterior ball condition, we can thus invoke \cite[Proposition 1.1]{MR3168912} to obtain that $\tilde p \in C^s(\mathbb{R}^n)$. The projection formula \eqref{eq:projection_control} thus yields $\tilde z \in C^s(\overline \Omega)$. Consequently, $\| \tilde z - \tilde z_h \|_{L^{\infty}(\Omega_h)} \rightarrow 0$ as $h \downarrow 0$. Invoke that $\tilde z$ is a global solution of  \eqref{eq:min}--\eqref{eq:weak_st_eq} and that $\bar z_h$ corresponds to a global solution of the discrete control problem to arrive at
\[
j(\tilde z) 
\leq 
j(\bar z) 
\leq 
\liminf_{h \downarrow 0} j_h( \bar z_h ) 
\leq  
\limsup_{h \downarrow 0} j_h( \bar z_h )
\leq  
\limsup_{h \downarrow 0} j_h( \tilde z_h )
=
j( \tilde z).
\]
To obtain the last equality, we used that $\| \tilde z - \tilde z_h \|_{L^{\infty}(\Omega_h)} \rightarrow 0$ implies $j_h(\tilde z_h) \rightarrow j(\tilde z)$ as $h \downarrow 0$. We have thus proved that $\bar z$ is a global solution and $j_h(\bar z_h) \rightarrow j(\bar z)$ as $h \downarrow 0$.

We now prove that $\| \bar z - \bar z_h \|_{L^2(\Omega)} \rightarrow 0$ as $h \downarrow 0$. In view of Proposition \ref{eq:proposition_convergence}, 
we have that $\bar u_h \rightarrow \bar u$ in $L^{\mathfrak{q}}(\Omega)$, for $\mathfrak{q} \leq 2n/(n-2s)$, as $h \downarrow 0$. Consequently,
\[
\left| \int_{\Omega}L(x,\bar u(x))\mathrm{d}x -\int_{\Omega_h} L(x, \bar u_h(x)) \mathrm{d}x  \right| 
\rightarrow 0, \quad h \downarrow 0.
\]
In view of the convergence result $j_h(\bar z_h) \rightarrow j(\bar z)$, we can thus obtain
\[
\tfrac{\alpha}{2} \| \bar z_h \|^2_{L^2(\Omega_h)} \rightarrow \tfrac{\alpha}{2} \| \bar z\|^2_{L^2(\Omega)},
\quad
h \downarrow 0.
\]
This and the weak convergence $\bar z_h \rightharpoonup \bar{z}$ in $L^{2}(\Omega)$ 
imply that $ \bar z_h \rightarrow \bar z$ in $L^2(\Omega)$ as $h \downarrow 0$. This concludes the proof.
\end{proof}

We now prove a somehow reciprocal result: \emph{every strict local minimum of the continuous problem \eqref{eq:min}--\eqref{eq:weak_st_eq} can be approximated by local minima of the discrete optimal control problems.}

\begin{theorem}[convergence]
Let the assumptions of Theorem \ref{thm:convergence} hold.
Let $\bar z$ be a strict local minimum of problem \eqref{eq:min}--\eqref{eq:weak_st_eq}. Then,  there exists a sequence $\{ \bar z_h \}_{h>0}$ of local minima of the discrete optimal control problems such that
\begin{equation}
\label{eq:convergence_local_minima}
\| \bar z - \bar z_h \|_{L^{2}(\Omega)} \rightarrow 0,
\qquad
j_{h}( \bar z_h) \rightarrow j(\bar z), 
\end{equation}
as $h \downarrow 0$.
\label{thm:convergence_local_minima}
\end{theorem}
\begin{proof}
Since $\bar z$ is a strict local minimum for problem \eqref{eq:min}--\eqref{eq:weak_st_eq}, we deduce the existence of $\epsilon >0$ such that the minimization problem
\begin{equation}
\min \{ j(z):  z \in \mathbb{Z}_{ad}~\mathrm{and}~\| \bar z - z \|_{L^2(\Omega)} \leq \epsilon \}
\label{eq:local_continuous_problem}
\end{equation}
admits a unique solution $\bar z \in \mathbb{Z}_{ad}$. On the other hand, let us introduce, for $h>0$, the discrete problem
\begin{equation}
\min \{ j_h(z_h):  z_h \in \mathbb{Z}_{ad,h}~\mathrm{and}~\| \bar z - z_h \|_{L^2(\Omega)} \leq \epsilon \}.
\label{eq:local_discrete_problem}
\end{equation}
To conclude that problem \eqref{eq:local_discrete_problem} admits at least a solution, we need to verify that the set where the minimum is sought is nonempty; notice that such a set is compact. To accomplish this task, we define, as in the proof of Theorem \ref{thm:convergence}, $\hat z_h \in \mathbb{Z}_{ad,h}$ by $\hat z_h |_T := \int_{T} \bar z(x) \mathrm{d}x / |T|$ for $T \in \T_h$. 
Since $\bar z \in C^s(\bar \Omega)$, we have that $\| \bar z - \hat z_h \|_{L^{\infty}(\Omega)} \rightarrow 0$ as $h \downarrow 0$. As a result, if $h$ is sufficiently small, $\hat z_h \in \mathbb{Z}_{ad,h}$ is such that $\| \bar z - \hat z_h \|_{L^2(\Omega)} \leq \epsilon$. We can thus conclude the existence of $h_{\star} > 0$ such that problem \eqref{eq:local_discrete_problem} admits at least a solution for $h \leq h_{\star}$.

Let $h \leq h_{\star}$ and let $\bar z_h$ be a global solution to problem \eqref{eq:local_discrete_problem}. Since $\{ \bar z_h \}_{0 < h \leq h_{\star}}$ is bounded in $L^{\infty}(\Omega)$, there exist a subsequence $\{ \bar z_{h_k} \}_{k= 1}^{\infty}$ of $\{ \bar z_h \}_{0 < h \leq h_{\star}}$ such that $\bar z_{h_k} \mathrel{\ensurestackMath{\stackon[1pt]{\rightharpoonup}{\scriptstyle\ast}}} \tilde{z}$ in $L^{\infty}(\Omega)$ as $k \uparrow \infty$. Proceed as in the proof of Theorem \ref{thm:convergence} to obtain that $\tilde z$ is a solution to the continuous problem \eqref{eq:local_continuous_problem} and $\bar z_{h_k} \rightarrow \tilde z$ in $L^2(\Omega)$ as $k \uparrow \infty$. Since problem \eqref{eq:local_continuous_problem} admits a unique solution, we must have $\tilde z = \bar z$ and $\bar z_h \rightarrow \bar z$ in $L^2(\Omega)$ as $h \downarrow 0$. Observe that, for $h$ sufficiently small, the constraint $\| \bar z - \bar z_h \|_{L^2(\Omega)} \leq \epsilon$ is not active in problem \eqref{eq:local_discrete_problem}. Consequently, $\bar z_h$ solves the original discrete problem. The remaining convergence property in \eqref{eq:convergence_local_minima} follows from the arguments elaborated in the proof of Theorem \ref{thm:convergence}. This concludes the proof.
\end{proof}

\begin{remark}[curved domains]\rm
The results of Theorems \ref{thm:convergence} and \ref{thm:convergence_local_minima} can also be obtained for curved domains. In fact, to prove Theorem \ref{thm:convergence} it suffices to consider a uniformly bounded extension $\{ \tilde z_h \}_{h>0} \subset L^{\infty}(\Omega)$ of $\{ z_h \}_{h>0} \subset L^{\infty}(\Omega_h)$. On the other hand, to adapt the results of Theorem \ref{thm:convergence_local_minima} to a cuved setting, we extend discrete functions $z_h \in \mathbb{Z}_{ad,h}$, defined over $\Omega_h$, to $\Omega$ by setting $z_h(x) = \bar z(x)$ for $x \in \Omega\setminus \Omega_h$.
\end{remark}

\subsection{Error estimates}
\label{sec:error_estimates}
Let $\{ \bar z_h \} \in \mathbb{Z}_{ad,h}$ be a sequence of local minima of the discrete optimal control problems such that $\| \bar z - \bar z_h \|_{L^{2}(\Omega_h)} \rightarrow 0$ as $h \downarrow 0$; $\bar z$ being a local solution of the continuous problem \eqref{eq:min}--\eqref{eq:weak_st_eq}; see Theorems \ref{thm:convergence} and \ref{thm:convergence_local_minima} . The main goal of this section is to provide an error estimate for $\bar z - \bar z_h$ in $L^2(\Omega_h)$, namely
\begin{equation}
\label{eq:error_estimate_in L2_aux}
\| \bar z - \bar z_{h} \|_{L^2(\Omega_h)}\lesssim h^{\gamma},
\quad
\gamma = \min \left\{ 1, s +  \tfrac{1}{2} - \epsilon \right\},
\quad
\forall h \leq h_{\star}.
\end{equation}
Here, $\epsilon >0$ is arbitrarily small. In what follows, if necessary, we extend discrete functions $z_h \in \mathbb{Z}_{ad,h}$, defined over $\Omega_h$, to $\Omega$ by setting $z_h(x) = \bar z(x)$ for $x \in \Omega\setminus \Omega_h$. 

We begin with the following instrumental result.
\begin{theorem}[instrumental error estimate]
Let $n \in \{ 2,3 \}$ and $s > n/4$. Let $\Omega$ be a convex domain such that $\partial \Omega \in C^{\infty}$. Assume that \ref{A1}--\ref{A3}, \ref{B1}--\ref{B2}, and \ref{C1}--\ref{C2} hold. Assume, in addition, that \eqref{eq:assumption_on_a_phi} and \eqref{eq:u_h_bounded} hold. Let $\bar z \in \mathbb{Z}_{ad}$ satisfies the second order optimality condition \eqref{eq:second_order_suff}, or equivalently \eqref{eq:second_order_equivalent}. Let us assume that \eqref{eq:error_estimate_in L2_aux} is false. Then, there exists $h_{\star}>0$ such that
\begin{equation}
\label{eq:basic_estimate}
\tfrac{\mathfrak{C}}{2}\| \bar z - \bar z_{h} \|^2_{L^2(\Omega_h)} \leq \left[ j'(\bar z_h) - j'(\bar z) \right](\bar z_h - \bar z)
\end{equation}
for every $h \leq h_{\star}$, where $\mathfrak{C} = \min \{ \mu, \alpha \}$, $\mu$ is the constant appearing in \eqref{eq:second_order_equivalent}, and $\alpha$ denotes the regularization parameter.
\label{thm:instrumental_error_estiamate}
\end{theorem}
\begin{proof}
Since \eqref{eq:error_estimate_in L2_aux} is false, there exist sequences $\{ h_{\ell} \}_{\ell=1}^{\infty}$ and $\{ \bar z_{h_{\ell}} \}_{\ell=1}^{\infty}$ such that $h_{\ell} \downarrow 0$ as $\ell \uparrow \infty$ and $\| \bar z - \bar z_{h_{\ell}} \|_{L^2(\Omega_h)}/h_{\ell}^{\gamma} \rightarrow \infty$ as $h_{\ell} \downarrow 0$. In what follows, to simplify notation we omit the subindex $\ell$. 

Define $v_h:= (\bar z_{h} - \bar z)/ \| \bar z_h - \bar z \|_{L^2(\Omega)}$ and observe that, for every $h >0$, we have $\| v_h \|_{L^2(\Omega)} = 1$. Upon considering a subsequence, if necessary, we can assume that $v_h  \rightharpoonup v$ in $L^2(\Omega)$ as $h \downarrow 0$. Since the set of elements satisfying \eqref{eq:sign_cond} is weakly closed in $L^2(\Omega)$ and each $v_{h}$ satisfies \eqref{eq:sign_cond}, we conclude that $v$ satisfies \eqref{eq:sign_cond} as well. We now prove that $|\bar{\mathfrak{p}}(x)| > 0$ implies $v(x) = 0$; recall that $\bar{\mathfrak{p}} = \bar p + \alpha \bar z$. Define $\bar{\mathfrak{p}}_h:= \bar p_h + \alpha \bar z_h$. Observe that Proposition \ref{eq:proposition_convergence} and the arguments elaborated in the proof of Theorems \ref{thm:error_estimates_adjoint_equation} and \ref{thm:error_estimates_adjoint_equation_2} yield $\| \bar{\mathfrak{p}}  - \bar{\mathfrak{p}}_h \|_{L^2(\Omega)} \rightarrow 0$ as $h \downarrow 0$. As a result, we can thus arrive at
\begin{multline*}
\int_{\Omega} \bar{\mathfrak{p}}(x) v(x) \mathrm{d}x = \lim_{h \rightarrow 0} \int_{\Omega_{h}}  \bar{\mathfrak{p}}_{h}(x) v_{h}(x) \mathrm{d}x 
\\
=  \lim_{h \rightarrow 0} \frac{1}{ \| \bar z_{h} - \bar z \|_{L^2(\Omega)} } \left[ \int_{\Omega_{h}} \bar{\mathfrak{p}}_{h}(x) [ (\Pi_{h} \bar z(x) - \bar z(x))
+  (\bar z_{h}(x) - \Pi_{h} \bar z(x)) ]
\mathrm{d}x
\right],
\end{multline*}
where $\Pi_h: L^2(\Omega) \rightarrow \mathbb{Z}_h$ denotes the orthogonal projection operator onto piecewise constant functions over $\T_h$. Since $0 \leq j_{h}'( \bar z_{h}) ( \Pi_{h} \bar z - \bar z_{h} ) = \int_{\Omega_h} \bar{\mathfrak{p}}_{h}(x)  (\Pi_{h} \bar z(x) - \bar z_{h}(x)) \mathrm{d}x$, because $\Pi_{h} \bar z \in \mathbb{Z}_{ad,h}$, we have
\[
\int_{\Omega} \bar{\mathfrak{p}}(x) v(x) \mathrm{d}x 
\leq
\lim_{h \rightarrow 0} \frac{1}{ \| \bar z_{h} - \bar z \|_{L^2(\Omega)} } \left[ \int_{\Omega_{h}} \bar{\mathfrak{p}}_{h}(x)(\Pi_{h} \bar z(x) - \bar z(x))
\mathrm{d}x \right].
\]
Invoke the regularity results for $\bar z$ obtained in Theorem \ref{thm:regularity_space}, namely $\bar z \in H^{s+1/2-\epsilon}(\Omega)$, where 
$\epsilon >0$ is arbitrarily small, standard error estimates for $\Pi_h$, and $\lim_{h \rightarrow 0} \| \bar z - \bar z_{h} \|_{L^2(\Omega_h)}/ h^{\gamma} = \infty$ to obtain
$
\int_{\Omega} \bar{\mathfrak{p}}(x) v(x) \mathrm{d}x \leq 0.
$
In view of \eqref{eq:sign_cond}, we can thus conclude that $\int_{\Omega} | \bar{\mathfrak{p}}(x) v(x)| \mathrm{d}x = 0$ and thus that $|\bar{\mathfrak{p}}(x)| > 0$ implies $v(x) = 0$ for a.e~$x \in \Omega$. Consequently, $v \in C_{\bar z}$.

Invoke the mean value theorem to obtain
\begin{equation}
\left( j'(\bar z_h) - j'(\bar z) \right) (\bar z_h - \bar z) = j''( \bar z + \theta_h ( \bar z_h - \bar z ) ) (\bar z_h - \bar z)^2, \quad \theta_h \in (0,1).
\label{eq:aux_second_order_2}
\end{equation}
Define $\hat{z}_h := \bar z + \theta_h ( \bar z_h - \bar z )$, $u_{\hat{z}_h} = u(\hat{z}_h):= \mathcal{S} \hat{z}_h$, and $p_{\hat{z}_h} = p(\hat{z}_h)$ as the solution to \eqref{eq:adj_eq} with $u$ replaced by $u_{\hat{z}_h} $. Invoke \eqref{eq:charac_j2} to obtain 
\begin{equation}
\begin{aligned}
\lim_{h \rightarrow 0} j''(\hat z_h) v_{h}^2 
&
= \lim_{h \rightarrow 0}
\int_{\Omega}
\left(
\frac{\partial^2 L}{\partial u^2}(x, u_{\hat{z}_h} )
\phi_{v_h}^2
-
p_{\hat{z}_h} \frac{\partial^2 a}{\partial u^2}(x,u_{\hat{z}_h})\phi_{v_h}^2
+ 
\alpha v_{h}^2 
\right)
\mathrm{d}x
\\
& = 
\alpha
+
\int_{\Omega}
\left(
\frac{\partial^2 L}{\partial u^2}(x, \bar u)
\phi_{v}^2
-
\bar p \frac{\partial^2 a}{\partial u^2}(x,\bar u)\phi_{v}^2
\right)
\mathrm{d}x,
\end{aligned}
\end{equation}
where we have used 
$
u_{\hat{z}_h} \rightarrow \bar u
$
and
$
p_{\hat{z}_h} \rightarrow \bar p 
$
in $\tilde H^s(\Omega) \cap L^{\infty}(\Omega)$ and $\phi_{v_h} \rightharpoonup \phi_{v}$ in $\tilde H^s(\Omega)$; the latter implies that $\phi_{v_h} \rightarrow \phi_{v}$ in $L^{\mathfrak{q}}(\Omega)$ as $k \uparrow \infty$ for $\mathfrak{q} < 2n/(n-2s)$; see the proof of Theorem \ref{thm:suff_opt_cond} for details. Since
$\bar z$ satisfies \eqref{eq:second_order_suff}, Theorem \ref{thm:equivalent_opt_cond} yields
\begin{equation*}
\lim_{h \rightarrow 0} j''(\hat z_h) v_{h}^2  = \alpha + j''(\bar z)v^2 - \alpha \| v \|^2_{L^2(\Omega)} \geq \alpha + (\mu - \alpha) \| v \|^2_{L^2(\Omega)}.
\end{equation*}
Since $\| v \|_{L^2(\Omega)} \leq 1$, we can thus conclude that $\lim_{h \rightarrow 0} j''(\hat z_h) v_{h}^2 \geq \mathfrak{C}$, where $\mathfrak{C} = \min \{ \mu, \alpha \}$. As a result, there exists $h_{\star} >0$ such that, for every $h \leq h_{\star}$, we have
$
j''(\hat z_h) v_{h}^2 \geq \mathfrak{C}/2. 
$

In view of \eqref{eq:aux_second_order_2}, we can finally derive \eqref{eq:basic_estimate} and conclude the proof.
\end{proof}

We now provide an error estimate for the difference $\bar z - \bar z_h$ in $L^2(\Omega_h)$.

\begin{theorem}[error estimate for approximation of a control variable]
Let $n \in \{ 2,3 \}$ and $s > n/4$. Let $\Omega$ be a convex domain such that $\partial \Omega \in C^{\infty}$. Assume that \ref{A1}--\ref{A3}, \ref{B1}--\ref{B2}, and \ref{C1}--\ref{C2} hold.  Assume, in addition, that \eqref{eq:assumption_on_a_phi} and \eqref{eq:u_h_bounded} hold. Let $\bar z \in \mathbb{Z}_{ad}$ satisfies the second order optimality condition \eqref{eq:second_order_suff}, or equivalently \eqref{eq:second_order_equivalent}. Then, there exist $h_{\star} >0$ such that
\begin{equation}
\label{eq:error_estimate_in L2}
\| \bar z - \bar z_{h} \|_{L^2(\Omega_h)}\lesssim h^{\gamma}, 
\quad
\gamma = \min \left\{ 1, s +  \tfrac{1}{2} - \epsilon \right \},
\quad
\forall h \leq h_{\star},
\end{equation}
where $\epsilon >0$ is arbitrarily small.
\label{thm:error_estimate_control}
\end{theorem}
\begin{proof}
We proceed by contradiction. Let us assume that \eqref{eq:error_estimate_in L2} is false so that we have at hand the instrumental error estimate of Theorem \ref{thm:instrumental_error_estiamate}.

We begin by observing that $j_h'(\bar z_h)(z_h - \bar z_h) \geq 0$ for every $z_h \in \mathbb{Z}_{ad,h}$ and $j'(\bar z)(\bar z_h - \bar z) \geq 0$. In view of these inequalities, we invoke \eqref{eq:basic_estimate} to obtain
\begin{equation}
\tfrac{\mathfrak{C}}{2}\| \bar z - \bar z_h \|_{L^2(\Omega_h)}^2 \leq [j_h'(\bar z_h) -j'(\bar z_h)](z_h - \bar z_h)  + j'(\bar z_h)(  z_h - \bar z) 
\label{eq:first_step}
\end{equation}
for every $z_h \in \mathbb{Z}_{ad,h}$. Let $\Pi_h: L^2(\Omega) \rightarrow \mathbb{Z}_h$ be the orthogonal projection operator onto piecewise constant functions over $\T_h$. Set $z_h = \Pi_h \bar z \in \mathbb{Z}_{ad,h}$ in \eqref{eq:first_step} to obtain
\[
\tfrac{\mathfrak{C}}{2}\| \bar z - \bar z_h \|_{L^2(\Omega_h)}^2 \leq [j_h'(\bar z_h) -j'(\bar z_h)]( \Pi_h \bar z - \bar z_h)  + j'(\bar z_h)( \Pi_h \bar z - \bar z)  =: \mathrm{I} + \mathrm{II}.
\]

We bound the term $\mathrm{II}$ as follows. First, standard properties of $\Pi_h$ reveal that 
\begin{multline}
\mathrm{II}_{\Omega_h} := (p_{\bar z_h} + \alpha \bar z_h, \Pi_h \bar z - \bar z)_{L^2(\Omega_h)} = ( p_{\bar z_h}, \Pi_h \bar z - \bar z)_{L^2(\Omega_h)}
\\
= (p_{\bar z_h} - \Pi_h p_{\bar z_h}, \Pi_h \bar z - \bar z)_{L^2(\Omega_h)} \lesssim h^{\gamma + \vartheta + \frac{1}{2}-\epsilon} | p_{\bar z_h} |_{H^{s+\frac{1}{2}-\epsilon}(\Omega)} | \bar z |_{H^{\gamma}(\Omega)},
\end{multline}
where $\vartheta = \min \{s,1/2-\epsilon \}$, $\gamma = \min \{ 1, s +  1/2 - \epsilon \}$, and $\epsilon >0$ is arbitrarily small. Here, $p_{\bar z_h} = p(\bar z_h)$ denotes the solution to \eqref{eq:adj_eq} with $u$ replaced by $\mathcal{S} \bar z_h$. Notice that Theorem \ref{thm:regularity_space} guarantees that $\bar z \in H^{\gamma}(\Omega)$.
On the other, on the basis of \ref{C1}--\ref{C2} and a bootstrap argument, Proposition \ref{pro:state_regularity_smooth} reveals that $p_{\bar z_h} \in H^{s+1/2-\epsilon}(\Omega)$, where $\epsilon >0$ is arbitrarily small. The remaining term $\mathrm{II}_{\Omega\setminus \Omega_h}$ vanishes:
\[
| \mathrm{II}_{\Omega\setminus \Omega_h}| = |(p_{\bar z_h} + \alpha \bar z_h, \Pi_h \bar z -\bar z   )_{L^2(\Omega\setminus \Omega_h)}| = 0.
\]

We now control $\mathrm{I}$. To accomplish this task, we first observe that
$
\mathrm{I} = (\bar p_h  - p_{\bar z_h} , \Pi_h \bar z - \bar z_h)_{L^2(\Omega)}.
$
Second, we split $\mathrm{I} = \mathrm{I}_{\Omega_h} + \mathrm{I}_{\Omega \setminus \Omega_h}$ and control $\mathrm{I}_{\Omega_h}$ as follows:
\begin{multline}
\mathrm{I}_{\Omega_h} := (\bar p_h  - p_{\bar z_h}, \Pi_h \bar z - \bar z_h)_{L^2(\Omega_h)} 
= (\bar p_h  - p_{\bar z_h}, \Pi_h (\bar z - \bar z_h) )_{L^2(\Omega_h)} 
\\
\lesssim 
\| \bar p_h  - p_{\bar z_h} \|_{L^2(\Omega)} \| \bar z - \bar z_h\|_{L^2(\Omega_h)}
\leq \frac{\mathfrak{C}}{4} \| \bar z - \bar z_h\|^2_{L^2(\Omega_h)}
+C h^{2(\vartheta + \frac{1}{2} - \epsilon)} | p_{\bar z_h} |^2_{H^{s+\frac{1}{2}-\epsilon}(\Omega)},
\end{multline}
where $\vartheta = \min \{ s, 1/2 - \epsilon \}$, $\epsilon >0$ is arbitrarily small and $C>0$. The control of $\| \bar p_h  - p_{\bar z_h} \|_{L^2(\Omega)}$ follows from Theorem \ref{thm:error_estimates_adjoint_equation}. Since discrete functions $z_h \in \mathbb{Z}_{ad,h}$ are extended to $\Omega$ upon setting $z_h(x) = \bar z(x)$ in $\Omega \setminus \Omega_h$, $\mathrm{I}_{\Omega \setminus \Omega_h} = 0$.

A collection of the derived estimates yields the bound $\| \bar z - \bar z_h \|_{L^2(\Omega)} \lesssim h^{\gamma}$, which is a contradiction. This concludes the proof.
\end{proof}

\begin{remark}[optimality]
\rm
The error estimate \eqref{eq:error_estimate_in L2} is, in terms of approximation, optimal for $s> \tfrac{1}{2}$ and suboptimal for $s \leq \tfrac{1}{2}$; suboptimality being dictated by the regularity properties obtained in Theorem \ref{thm:regularity_space}.
\end{remark}
\begin{remark}[$L^2(\Omega_h)$-error estimate]
\rm
To derive \eqref{eq:error_estimate_in L2}, the $L^2(\Omega_h)$-error estimate of Theorem \ref{thm:error_estimates_adjoint_equation} and the instrumental one obtained in Theorem \ref{thm:instrumental_error_estiamate} are essential. The latter strongly relies on the second order optimality conditions analyzed in Theorems \ref{thm:suff_opt_cond} and \ref{thm:equivalent_opt_cond}; observe inequality \eqref{eq:second_order_equivalent}. In view of the assumptions $n \in \{2,3\}$ and $s>n/4$, solutions to the state and adjoint equations belong to $\tilde H^s(\Omega) \cap L^{\infty}(\Omega)$ for corresponding forcing terms in $L^2(\Omega)$. Without these assumptions, global ones on $a$ and $L$ should be imposed in order to provide an analysis and an error estimate in $L^2(\Omega_h)$; observe the local nature in \ref{A1}--\ref{A3}, \ref{B1}--\ref{B2}, and \ref{C1}--\ref{C2}.
\end{remark}

We conclude the section with the following result.

\begin{theorem}[error estimates for approximation of state and adjoint variables]
Let the assumptions of Theorem \ref{thm:error_estimate_control} hold. Assume, in addition, that $\partial L/\partial u$ is locally Lipschitz with respect to the second variable. Then, there exist $h_{\star} >0$ such that
\begin{equation}
\label{eq:error_estimate_in L2_state_adjoint}
\| \bar u - \bar u_{h} \|_{s} \lesssim h^{\frac{1}{2}-\epsilon}, 
\qquad
\| \bar p - \bar p_{h} \|_{s} \lesssim h^{\frac{1}{2}-\epsilon},
\qquad
\forall h \leq h_{\star},
\end{equation}
where $\epsilon >0$ is arbitrarily small.
\label{thm:error_estimate_state_adjoint}
\end{theorem}
\begin{proof}
Observe that $\| \bar u - \bar u_{h} \|_{s} \leq \| \bar u - y \|_{s} + \| y - \bar{u}_h \|_{s}$, where $y$ solves \eqref{eq:mathsf_u} with $z_h$ replaced by $\bar{z}_h$. The assumptions on $a$ combined with the arguments elaborated in the proof of \cite[Theorem 4.16]{MR2583281} and the error estimate \eqref{eq:error_estimate_semilinear_s_final} yield the estimates
\[
\| \bar u - y  \|_{s} \lesssim \| \bar z - \bar z_h \|_{L^2(\Omega)},
\qquad
\| y - \bar{u}_h \|_{s} \lesssim h^{\frac{1}{2}- \epsilon} \| y \|_{H^{s+\frac{1}{2}-\epsilon}(\Omega)}.
\]
In view of \eqref{eq:error_estimate_in L2}, we obtain the left-hand side estimate in \eqref{eq:error_estimate_in L2_state_adjoint}. The bound for the error committed within the approximation of $\bar p$ is the content of Theorem \ref{thm:error_estimates_adjoint_equation_2}.
\end{proof}

\section{Numerical experiments}
In this section, we illustrate the performance of the fully discrete scheme proposed in section \ref{sec:fem_control} and the sharpness of the error estimates derived in Theorems \ref{thm:error_estimate_control} and \ref{thm:error_estimate_state_adjoint}.

\subsection{Implementation}
\label{sec:implementation}
The implementation has been carried out in MATLAB on the basis of the finite element code devised in \cite{MR3679865}. To solve the discrete optimality system we have utilized a slight variation of the primal--dual active set strategy described in \cite[Section 2.12.4]{MR2583281}; on each iteration of such a strategy we solve the ensuing nonlinear system by using a Newton--type method.

Let us now construct an exact solution for the \emph{fractional semilinear optimal control problem} \eqref{eq:min}--\eqref{eq:weak_st_eq}. To accomplish this task, we slight modify the state equation \eqref{eq:state_equation} by incorporating an extra forcing term $f$: $(-\Delta)^s u + a(\cdot, u) = f + z$ in $\Omega$ and $u = 0$ in $\Omega^c$. Let $n=2$, $\Omega = B(0,1)$, $s \in (0,1)$, and $\alpha = 1$. We set the control bounds $\mathfrak{a}$ and $\mathfrak{b}$, the nonlinear term $a(\cdot,u)$, and the forcing term $f$ to be as follows:
\[
\mathfrak{a} = -0.8,
\qquad 
\mathfrak{b} = -0.1,
\qquad
a(\cdot,u) = u^3,
\qquad
f = 1 - \bar{u}^3 - \bar{z}.
\] 
Within this setting, the optimal state $\bar u$ is given by
\begin{equation}
\label{eq:example}
\bar {u}(x) = \frac{\Gamma(\frac{n}{2})}{2^{2s}\Gamma(\frac{n+2s}{2})\Gamma(1+s)} \left( 1 - |x|^{2}\right)^s_{+},
\quad
t^{+} = \max \{ t,0\}, \quad n = 2.
\end{equation}
Additionally, we set $L(\cdot,u) = \tfrac{1}{2}(u-u_d)^2$, where $u_d$ corresponds to a suitable \emph{desired state}. Within this particular framework, the optimal adjoint variable $\bar{p}$ satisfies
\begin{equation*}
\label{eq:adj_eq_numerics}
\bar{p} \in \tilde{H}^s(\Omega):
\quad
\mathcal{A}(v,\bar{p}) + \left(3\bar{u}^2 \bar{p},v\right)_{L^2(\Omega)} = \left(\bar{u}-u_d,v\right)_{L^2(\Omega)}
\quad \forall v \in \tilde H^s(\Omega).
\end{equation*}
We consider $u_d = \bar{u} - 3 \bar{u}^2 \bar{p} -1$ so that $\bar p$ is given by \eqref{eq:example}. Finally, the optimal control can be obtained in view of a projection formula: $\bar z(x) = \Pi_{[\mathfrak{a},\mathfrak{b}]}(-\bar p(x))$ for a.e.~$x \in \Omega$.

We notice that the previous construction of an exact solution is such that both $\bar u$ and $\bar p $ belong to $H^{s+1/2-\epsilon}(\Omega)$, for every $\epsilon >0$, so that it retains essential difficulties and singularities and allows us to evaluate experimental rates of convergence.

\subsection{Experimental rates of convergence}
We discretize $\Omega$ using a sequence of quasi-uniform meshes and study the performance of the fully discrete scheme described in section \ref{sec:fem_control} with $s \in \{0.1, 0.2, \cdots, 0.9\}$. In Figure \ref{fig:1} we present experimental rates of convergence for the error committed in the approximation of the optimal control, state, and adjoint variables. We observe that such experimental rates of convergence are in agreement with the error estimates derived in Theorems \ref{thm:error_estimate_control} and \ref{thm:error_estimate_state_adjoint}.

\begin{figure}[!ht]
\centering
\psfrag{s = 01}{{\normalsize $s = 0.1$}}
\psfrag{s = 02}{{\normalsize $s = 0.2$}}
\psfrag{s = 03}{{\normalsize $s = 0.3$}}
\psfrag{s = 04}{{\normalsize $s = 0.4$}}
\psfrag{s = 05}{{\normalsize $s = 0.5$}}
\psfrag{s = 06}{{\normalsize $s = 0.6$}}
\psfrag{s = 07}{{\normalsize $s = 0.7$}}
\psfrag{s = 08}{{\normalsize $s = 0.8$}}
\psfrag{s = 09}{{\normalsize $s = 0.9$}}
\psfrag{ndof -03}{{\normalsize $\mathsf{Ndof}^{-0.3}$}}
\psfrag{ndof -035}{{\normalsize $\mathsf{Ndof}^{-0.35}$}}
\psfrag{ndof -04}{{\normalsize $\mathsf{Ndof}^{-0.4}$}}
\psfrag{ndof -045}{{\normalsize $\mathsf{Ndof}^{-0.45}$}}
\psfrag{ndof -05}{{\normalsize $\mathsf{Ndof}^{-0.5}$}}
\psfrag{ndofs 05}{{\normalsize $\mathsf{Ndof}^{-0.5}$}}
\psfrag{ndofs -025}{{\normalsize $\mathsf{Ndof}^{-0.25}$}}
\psfrag{control}{{\normalsize Control }}
\psfrag{State}{{\normalsize State }}
\psfrag{Ad. state}{{\normalsize Adjoint }}
\begin{minipage}[c]{0.45\textwidth}\centering
\psfrag{control error}{\large{$\|\bar{z} - \bar{z}_h \|_{L^{2}(\Omega)}$}}
\includegraphics[trim={0 0 0 0},clip,width=4.5cm,height=3.8cm,scale=0.40]{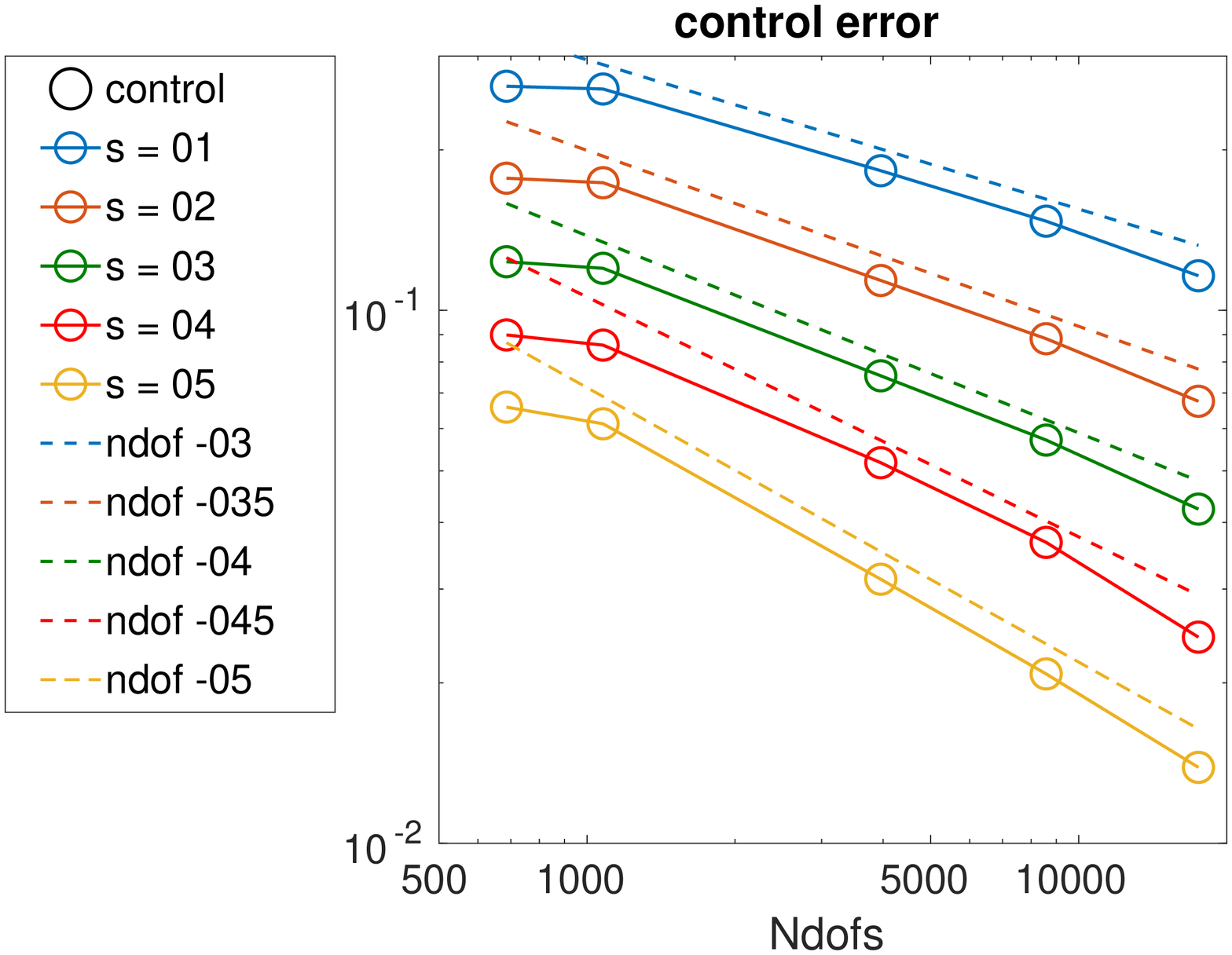}
\end{minipage}
\begin{minipage}[c]{0.45\textwidth}\centering
\psfrag{control error}{\large{$\|\bar{z} - \bar{z}_h \|_{L^{2}(\Omega)}$}}
\includegraphics[trim={0 0 0 0},clip,width=4.6cm,height=3.8cm,scale=0.40]{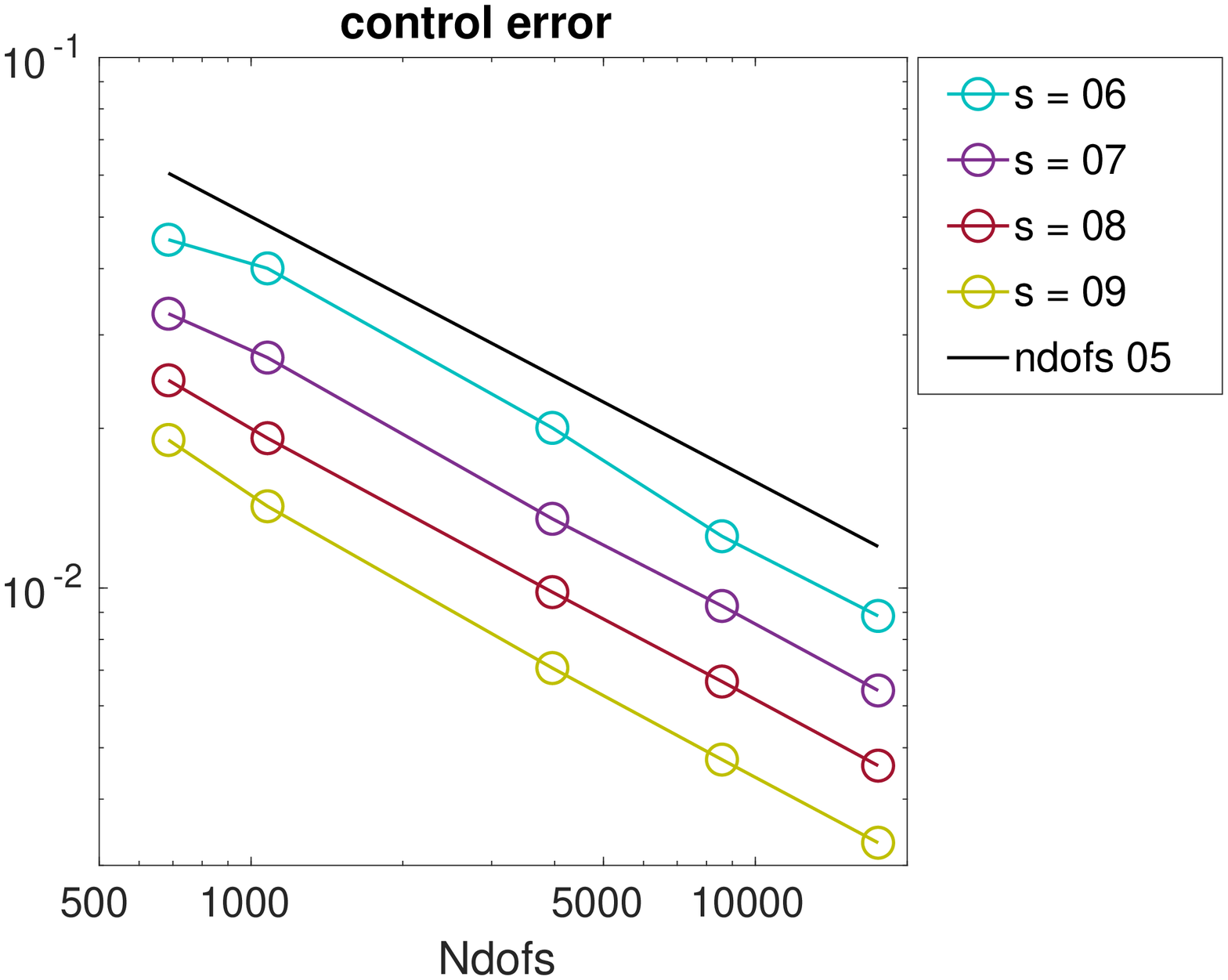}
\end{minipage}
\\[0.25cm]
\begin{minipage}[c]{0.45\textwidth}\centering
\psfrag{state error}{\hspace{0.2cm}\large{$\| \bar{u} - \bar{u}_h \|_s$}}
\includegraphics[trim={0 0 0 0},clip,width=4.5cm,height=4.0cm,scale=0.40]{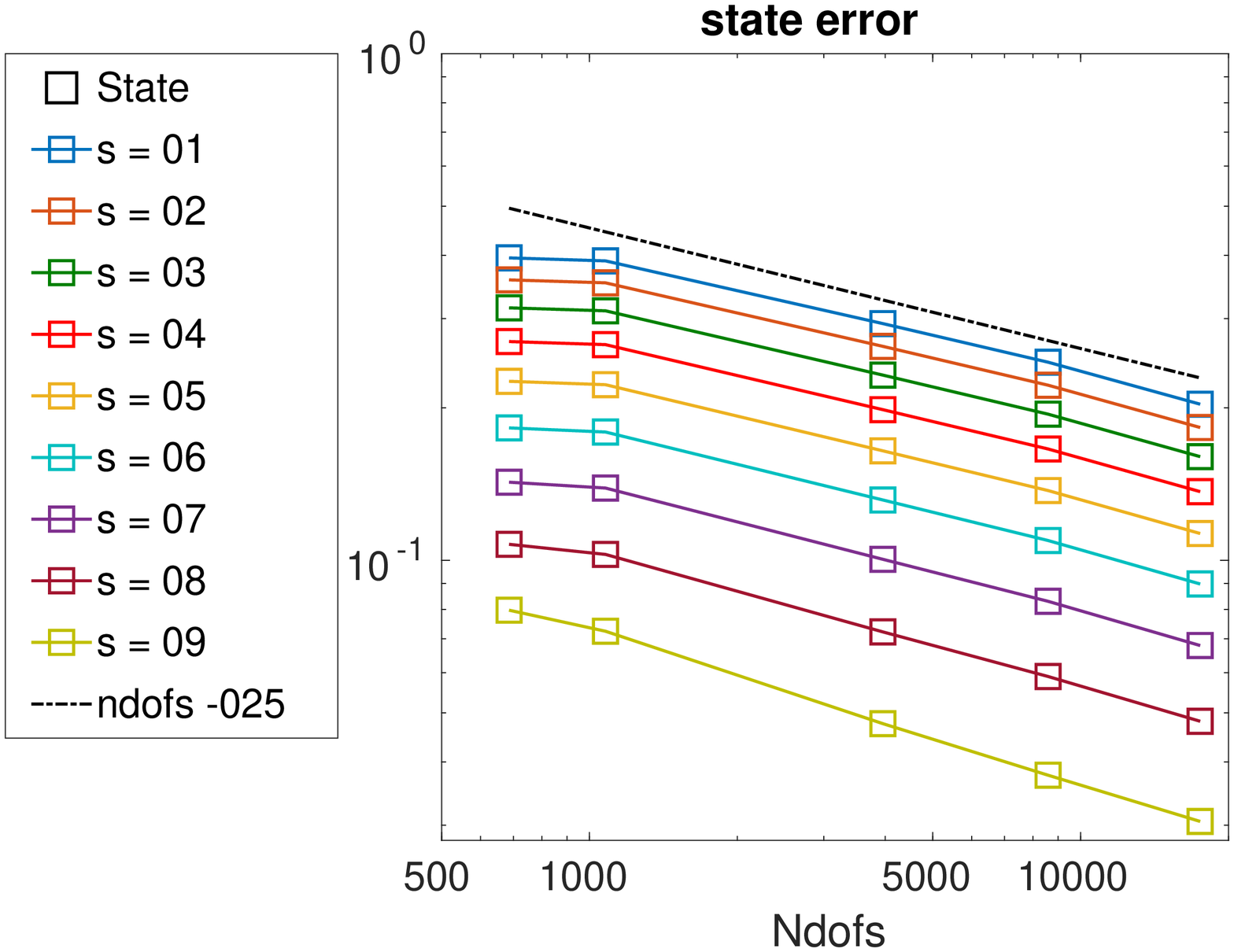}
\end{minipage}
\begin{minipage}[c]{0.45\textwidth}\centering
\psfrag{adjoint error}{\hspace{0.2cm}\large{$\| \bar{p} - \bar{p}_h \|_s$}}
\includegraphics[trim={0 0 0 0},clip,width=4.6cm,height=4.0cm,scale=0.40]{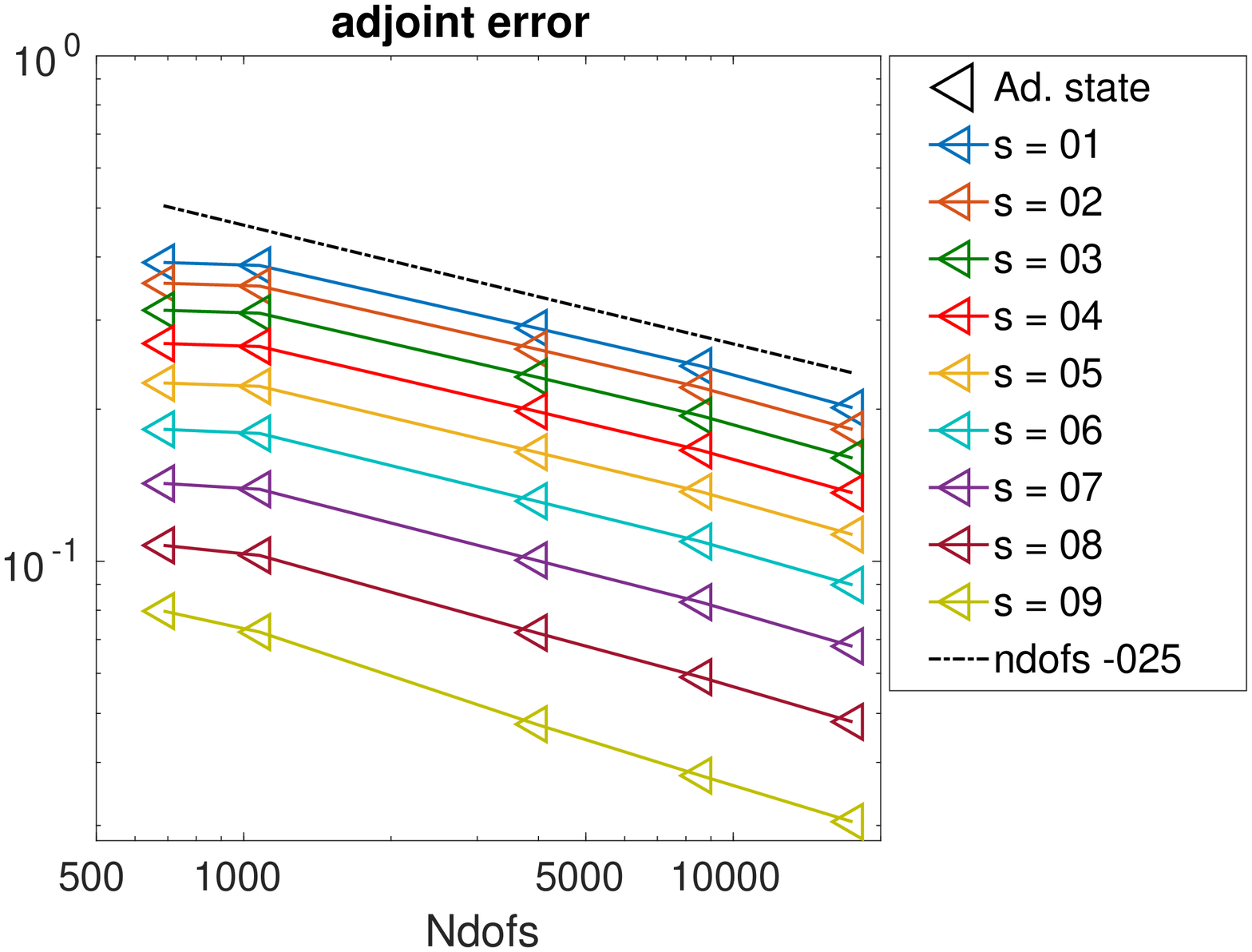}
\end{minipage}
\caption{\small{Experimental rates of convergence (ERO) for the fully discrete scheme of section \ref{sec:fem_control} within the setting described in section \ref{sec:implementation}. The obtained ERO are in agreement with the error estimates derived in Theorems \ref{thm:error_estimate_control} and \ref{thm:error_estimate_state_adjoint}; observe that $h \approx \mathsf{Ndof} ^{-1/2}$.}}
\label{fig:1}
\end{figure}

\bibliographystyle{plain}
\bibliography{bibliography}

\end{document}